\documentclass{mod}
\usepackage{url}
\usepackage{setspace}
\usepackage{scrextend}
\usepackage{cite}

\usepackage{amsmath}
\usepackage{amsthm}
\usepackage{amssymb}

\usepackage{mathtools}
\usepackage{mathrsfs}

\usepackage[all]{xy}
\usepackage[toc,page]{appendix}
\usepackage{titlesec}

 \newcounter{mainthm}

 \newtheorem{thm}{Theorem}[section]
 \newtheorem{lem}[thm]{Lemma}
 \newtheorem{prop}[thm]{Proposition}
 \newtheorem{cor}[thm]{Corollary}

 \theoremstyle{definition}

 \theoremstyle{remark}
 \newtheorem{rmk}[thm]{Remark}

\newtheorem{defn-thm}[thm]{Definition-Theorem}

\newtheorem{defn-lem}[thm]{Definition-Lemma}
\newtheorem{defn-prop}[thm]{Definition-Proposition}

\newcommand{\id}{\mathrm{id}}

\newcommand{\mathds}[1]{\text{\usefont{U}{dsrom}{m}{n}#1}}
\newcommand{\one}{\mathds {1}}

\newcommand{\pr}{\mathrm{pr}}
\newcommand{\ev}{\mathrm{ev}}

\newcommand{\uu}{\mathbf u}

\newcommand{\mg}{\mathbf g}

\newcommand{\eff}{\mathrm{eff}}

\newcommand{\m}{{\mathfrak m}}

\newcommand{\mC}{\mathfrak C}

\newcommand{\oi}{{[0,1]}}

\newcommand{\f}{\mathfrak f}

\newcommand{\g}{\mathfrak g}
\newcommand{\G}{\mathfrak G}

\newcommand{\mJ}{{\mathbf J}}

\DeclareMathOperator{\Sp}{Sp}

\DeclareMathOperator{\diff}{Diff}

\DeclareMathOperator{\val}{val}

\DeclareMathOperator{\trop}{\mathfrak{trop}}
\DeclareMathOperator{\Hom}{Hom}

\DeclareMathOperator{\im}{im}

\usepackage{subcaption}
\usepackage{indentfirst}
\titleformat{\paragraph}[runin]{\small\sffamily\bfseries}{}{}{}[]

\titleformat{\subsubsection}[runin]{\itshape\normalsize}{\thesubsubsection \ }{0em}{}[\mbox{ . } ]
\titlespacing{\subsubsection}
{0pt}
{3.25ex plus 1ex minus .2ex}
{0pt}

\DeclareMathAlphabet{\mathbbm}{U}{bbm}{m}{n}

\begin{document}
	\setlength{\parindent}{15pt}	\setlength{\parskip}{0em}

	\title{Disk counting and wall-crossing phenomenon via family Floer theory}
	\author[Hang Yuan]{Hang Yuan}
\address{
Email: \emph{hyuan@northwestern.edu}}
	\begin{abstract} {\sc Abstract:}  
		We use the wall-crossing formula in the non-archimedean SYZ mirror construction \cite{Yuan_I_FamilyFloer} to compute the Landau-Ginzburg superpotential and the one-pointed open Gromov-Witten invariants for a Chekanov-type Lagrangian torus in any smooth toric Fano compactification of $\mathbb C^n$. It agrees with the works of Auroux, Chekanov-Schlenk \cite{AuTDual, Chekanov_Schlenk}, and Pascaleff-Tonkonog \cite{PT_mutation}.
	\end{abstract}
	\maketitle
	%
	%
	

%

\section{Introduction}

One goal of the family Floer program is to achieve the B-side mirror reconstruction merely in terms of the Floer-theoretic data in the A-side without a priori knowledge of what the mirror is. It is first initiated by Fukaya \cite{FuFamily} and developed by Abouzaid \cite{AboFamilyICM, AboFamilyFaithful, AboFamilyWithout} and Tu \cite{Tu}.
It is also compatible with the Strominger-Yau-Zaslow's T-duality philosophy \cite{SYZ} and the Kontsevich-Soibelman's non-archimedean mirror symmetry \cite{KSTorus, KSAffine}.
Recently, the author further implements the non-archimedean SYZ construction in the family Floer program over the smooth locus with far less restrictive assumptions, including the full quantum corrections of holomorphic disks \cite{Yuan_I_FamilyFloer}.

In this paper, we show that the improvement to include quantum corrections actually leads to concrete enumerative geometry.
The outcomes perfectly accord with the previous works of Auroux \cite{AuTDual}, Chekanov-Schlenk \cite{Chekanov_Schlenk}, and Pascaleff-Tonkonog \cite{PT_mutation}, offering the examples and applications to strongly support the intricate theory in \cite{Yuan_I_FamilyFloer}.
Indeed, as long as we were familiar with the non-archimedean SYZ mirror construction therein and the topological aspects of the Gross's fibration (reviewed in Appendix \ref{S_nonarchimedean_review} and \ref{S_Gross_fibration} respectively), the proof of the main result will be extremely short and concise! But as a cost, the Floer-theoretic foundation in \cite{Yuan_I_FamilyFloer} is lengthy and complicated.

Specifically, we consider a Gross's special Lagrangian fibration \cite{Gross_ex} on $X=\mathbb C^n$ which only has two chambers $B_\pm$ of Clifford and Chekanov tori \cite{Chekanov1996LagrangianTI, eliashberg1997problem}.
The main result --- {Theorem \ref{Main_theorem_superpotential_compute_thm}} --- computes the superpotential for a Chekanov-type torus in any smooth toric Fano compactification $\overline X$ of $X$, such like $\overline X=\mathbb {CP}^n, \mathbb {CP}^r\times \mathbb {CP}^{n-r}$.
In fact, if let $W_\pm$ and $\overline W_\pm$
denote the Clifford/Chekanov superpotentials before and after the compactification, then
the Chekanov superpotential $\overline W_-$ for $\overline X$ is only known for a few examples, while the other three are all previously known \cite{Cho_Oh, CLL12}.

A major advance in \cite{Yuan_I_FamilyFloer} is a much more precise understanding of the wall-crossing phenomenon \cite{AuTDual}.
This is vital for the proof of Theorem \ref{Main_theorem_superpotential_compute_thm}.
For instance, we know there always exists a gluing map $\phi$, explicitly described by the $A_\infty$ structures and exclusively contributed by Maslov-zero disks, that matches the superpotentials in any two adjacent chambers (\S \ref{sss_family_Floer_general}).
In our case, a key observation for the Gross's fibration is that when we compactify $X$ to $\overline X$, the Maslov-zero disks stay the same, while there are extra Maslov-two disks included.
Thus, the wall-crossing formula \cite{Yuan_I_FamilyFloer} deduces that the same gluing map must satisfy both $\phi(W_+)=W_-$ and $\phi(\overline W_+)=\overline W_-$.
The first relation with the $T^{n-1}$-symmetry of the Gross's fibration determines $\phi$, and then the second one can be used to compute $\overline W_-$ explicitly.

For what it is worth, our result seems support a recurring point of view in symplectic topology: the `deformation' of the Floer theory on an open manifold $X$ implies the Floer theory on its compactification $\overline X$. This interplay has been a propulsive force behind numerous important developments: the proofs of homological mirror symmetry for the genus-two curve and the quadratic surface by Seidel \cite{Seidel03, Seidel_deformation, Seidel_genus2} and for projective hypersurfaces by Sheridan \cite{Sheridan15, Sheridan16}, the work of Ganatra-Pomerleano \cite{Ganatra_2016logPSS}, etc. For example, the Tonkonog's work \cite{Tonkonog_From_SH} suggests our work may be related to the ongoing work of Borman-Sheridan-Varolgunes \cite{Borman_Sheridan} via the closed-open maps; compare also \cite[Lemma 2.7]{Sheridan16}.
Moreover, our computation fits well with Pascaleff-Tonkonog's result \cite{PT_mutation} despite of the different method (Remark \ref{PT_rmk}). Indeed, they oppositely avoid any studies of Maslov index zero disks and apply the Seidel's ideas in \cite{Seidel_dynamics} instead.
We do not understand the whole story behind this coincidence at present, but it is certainly very fascinating to explore in the future.

\subsection{Notations}
\label{ss_Notations}
Let $X=\mathbb C^n$.
Fix $\epsilon>0$, and we consider the following special Lagrangian fibration:
\[
\hat\pi: X\to B, \quad  (z_1,\dots, z_n)\mapsto 
\big(
\tfrac{1}{2} (|z_1|^2-|z_n|^2), \dots, \tfrac{1}{2} (|z_{n-1}|^2-|z_n|^2), |z_1z_2\cdots z_n-\epsilon|-\epsilon
\big)
\]
It is an example of a \textit{Gross's fibration} \cite{Gross_ex}.
The general aspects of Gross's fibrations should be mostly standard. But we have included an exposition in \S \ref{S_Gross_fibration} in great details, and there are also some new perspectives.
The Gross's fibration keeps a $T^{n-1}$-symmetry instead of a $T^n$-symmetry. For instance, when $X=\mathbb C^n$, we have a fiber-preserving $T^{n-1}$-action defined by $(e^{i\theta_1},\dots, e^{i\theta_{n-1}}) \cdot (z_1,\dots,z_{n-1}, z_n)=  (e^{i\theta_1}z_1,\dots, e^{i\theta_{n-1}}z_{n-1}, e^{-i(\theta_1+\cdots+\theta_{n-1})} z_n)$.

The base of $\hat\pi$ is $B= \mathbb R^{n-1} \times [-\epsilon, +\infty)$; the discriminant locus is $\Gamma:=\partial B\cup ( \Pi\times\{0\})$ where $\Pi$ is the tropical hyperplane that consists of those $(\lambda_1,\dots, \lambda_{n-1})\in\mathbb R^{n-1}$ such that either $\min(\lambda_i)$ is attained twice or $\min(\lambda_i)=0$ (see \cite{AuSurvey}).
The smooth locus is $B_0 :=B\setminus \Gamma$, and we denote by 
\begin{equation}
\label{Gross_fibration_begin}
\pi:X_0\to B_0
\end{equation}
the restriction of $\hat\pi$ over $B_0$, where $X_0:=\hat\pi^{-1}(B_0)\subset X$.
It is a special Lagrangian torus fibration with respect to the holomorphic $n$-form $\Omega=(z_1\cdots z_n-\epsilon)^{-1} dz_1\wedge \cdots \wedge dz_n$.
A special Lagrangian fiber of $\pi$ is called a \textit{Gross's fiber}.
Define the divisors $D_i=\{z_i=0\}$ ($1\le i\le n$) and $\mathscr E=\{z_1z_2\cdots z_n=\epsilon\}$ in $X=\mathbb C^n$.
We define $H: =(\mathbb R^{n-1}\setminus \Pi)\times \{0\}$ and call it the \textit{wall}.
The tropical hyperplane $\Pi$ separates the wall  into $n$ different connected components $H_i$ ($1\le i\le n$), where the label $i$ is chosen so that $H_i\subset \pi(D_i)$. 
For $n=3$, one can check that $H_1=\{\lambda_1<0, \lambda_1<\lambda_2\}$, $H_2=\{\lambda_2<0, \lambda_1>\lambda_2\}$, and $H_3=\{\lambda_1>0,\lambda_2>0\}$.
A torus fiber $L_q$ over $q=(q_1,q_2)\in B_0$ bounds non-trivial Maslov index zero holomorphic disks if and only if $q_2=0$, i.e. $q\in H$ (Lemma \ref{Maslov_zero_locate_lem}).

\begin{figure}
	\centering
	\begin{subfigure}{0.4\textwidth}
		\includegraphics[scale=0.24]{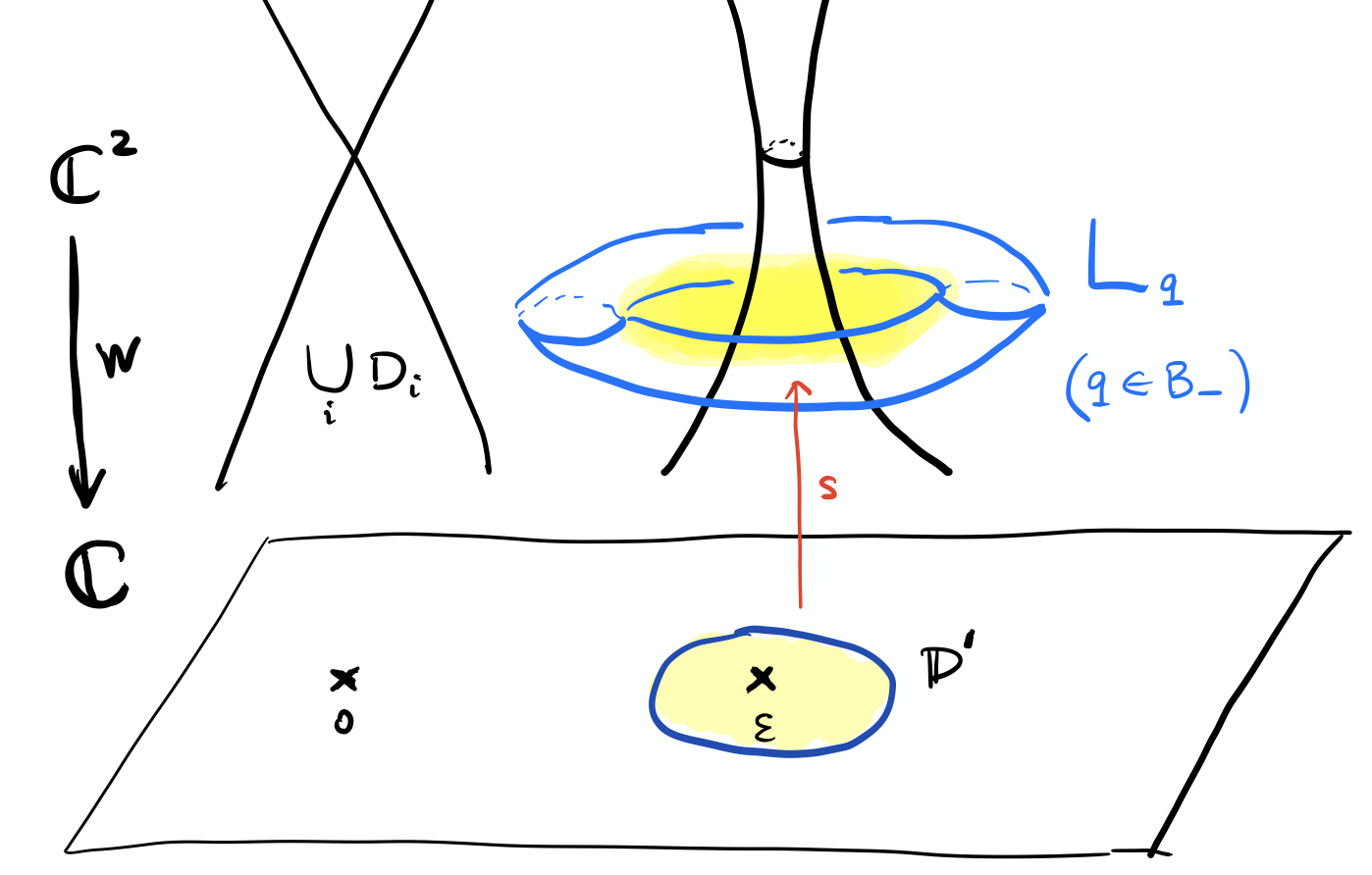}
		\caption{$X=\mathbb C^2$ with the visualization in \cite{AuTDual}. The shadowed disk represents $\hat\beta$.}
	\end{subfigure}
	\begin{subfigure}{0.55\textwidth}
		\includegraphics[scale=0.35]{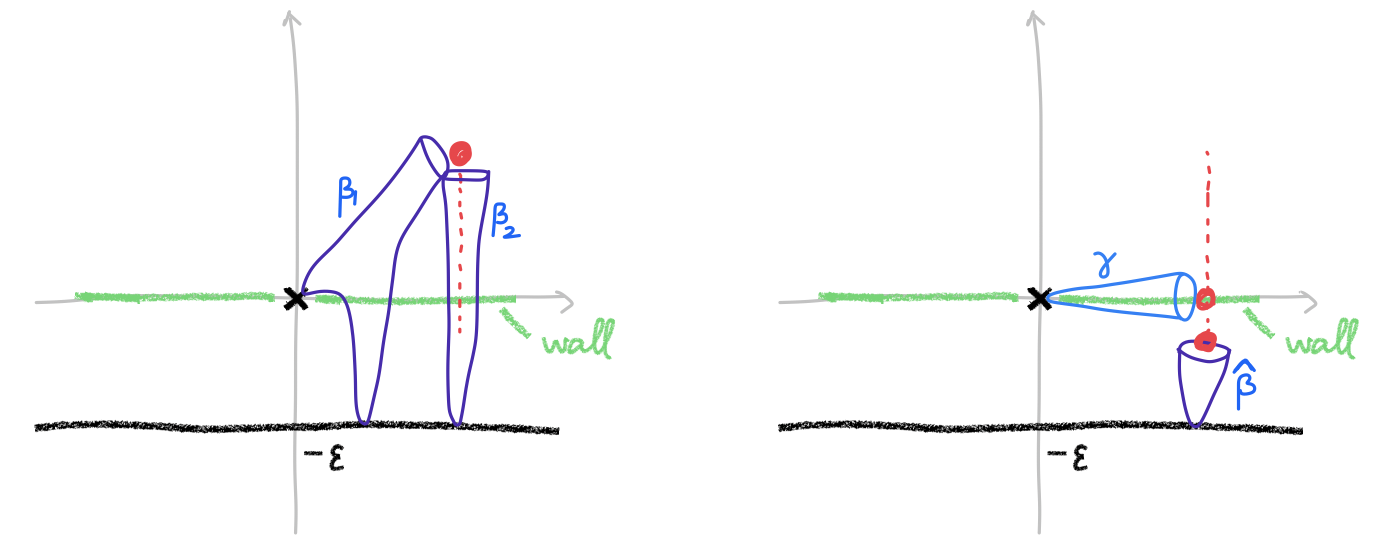}
		\caption{An illustration of (\ref{wall_cross_beta_Cn_eq}) when $X=\mathbb C^2$. $B_+=\mathbb R\times (0,+\infty)$ and $B_-=\mathbb R\times (-\epsilon,0)$.}
		\label{figure_B_-_monodromy_class}
	\end{subfigure}
	\caption{The holomorphic disk in $\hat\beta$ and the related wall-crossing}
	\label{figure_B_-_class}
\end{figure}

Let $N=\mathbb Z \{e_1,\dots, e_n\}$ and $M=\mathbb Z \{e'_1,\dots, e'_n\}$ be two dual lattices of rank $n$. We choose a fan structure $\Sigma$ for $X=\mathbb C^n$ to be the one spanned by $-e_1,\dots, -e_n$. Consider a smooth and complete extension $\overline \Sigma$ of $\Sigma$ with additional rays $v_1,\dots,v_m$. The number of rays in $\overline \Sigma$ is $n+m$.
Let $\overline X$ be the toric variety associated to the fan $\overline \Sigma$, then, it is a smooth toric compactification of $X=\mathbb C^n$. We further require $\overline X$ is Fano. For example, when $m=1$ and $v_1=e_1+\cdots +e_n$, we have $\overline X=\mathbb {CP}^n$; when $m=2$, $v_1=e_1+\cdots+e_r$, and $v_2=e_{r+1}+\cdots+ e_n$, we have $\overline X=\mathbb {CP}^r\times \mathbb {CP}^{n-r}$.
Next, we define
\[
B_+=\{q=(q_1,q_2)\mid q_2>0\}=\mathbb R^{n-1}\times (0,+\infty)
\]
\[
B_-=\{q=(q_1,q_2)\mid q_2<0\} =\mathbb R^{n-1}\times (-\epsilon,0)
\]
Then, $B_0=B_+\sqcup H\sqcup B_0$.
We call the $B_+$ (resp. $B_-$) the chamber of Clifford tori (resp. Chekanov tori).
No matter in $X$ or $\overline X$, we say the torus fiber $L_q$ is of \textit{Clifford type} if $q_2>0$ and it is of \textit{Chekanov type} if $q_2<0$. Here we follow the terms in \cite{AuTDual}; the second kind is studied in \cite{Chekanov1996LagrangianTI, eliashberg1997problem}.

Let $D_1,\dots, D_n, D'_1,\dots, D_m'$ be the irreducible toric divisors associated to the rays $-e_1,\dots, -e_n$, $v_1,\dots, v_m$.
Note that $X$ embeds into $\overline X$, and
$\overline D:=\bigcup_a D'_a = \overline X\setminus X$.
The closures of $D_i$ and ${\mathscr E}$ in $\overline X$ are still denoted by $D_i$ and $\mathscr E$ respectively.
For a fiber $L$, consider the following natural exact sequences:
\begin{equation}
\label{exact_seq_pi_2_X_bar}
0\to \pi_2(\overline X)\to \pi_2(\overline X, L)\xrightarrow{\partial} \pi_1(L)\to 0
\end{equation}
\begin{equation}
\label{exact_seq_Pic}
0\to \mathrm{Pic}(\overline X)^\vee \to \mathbb Z^{n+m} \to N \to 0
\end{equation}
Notice that $H_2(\overline X;\mathbb Z)\cong \pi_2(\overline X)\cong \mathrm{Pic}(\overline X)^\vee$ and $\mathrm{Pic}(\overline X)\cong \mathrm{Span}\{D_1',\dots,D'_m\}$.

If $L$ is of Clifford type, it can deform into a product torus by a Hamiltonian isotopy inside $(\mathbb C^*)^n$. Thus, there is a canonical isomorphism $\pi_1(L)\cong N$; also, $\pi_2(\overline X, L)$ is naturally isomorphic to $\mathbb Z^{n+m}$ via $\beta\mapsto (\beta \cdot D_i \ ; \beta\cdot D'_a)$.
In reality, we have the Maslov-two classes $\beta_1,\dots,\beta_n$ and $\beta'_1,\dots,\beta'_m$ in $\pi_2(\overline X,L)$ such that $\partial\beta_i = -e_i$ for $1\le i \le n$ and $\partial\beta'_a=v_a$ for $1\le a\le m$.
Also,
$\beta_i\cdot D_j=\delta_{ij}$, $\beta_a'\cdot D_b'=\delta_{ab}$, and $\beta_i\cdot D'_b=\beta'_a\cdot D_j=0$. Besides, $\beta_i\cdot \mathscr E=1$ and $\beta'_a\cdot \mathscr E=0$. By \cite{Cho_Oh}, their open Gromov-Witten invariants are all equal to one: $\mathsf n_{\beta_i}=\mathsf n_{\beta'_a}=1$.
On the other hand, the sequence (\ref{exact_seq_pi_2_X_bar}) is actually splitting. Given $1\le a\le m$, there exist $v_{a1},\dots, v_{an}\in\mathbb Z$ such that $\partial\beta'_a=v_a=v_{a1} e_1+v_{a2}e_2+\cdots+ v_{an} e_n\in N$, so we can find some $\mathcal H_a\in \pi_2(\overline X)$ with 
\begin{equation}
\label{H_a_eq}
\mathcal H_a=\beta_a'+v_{a1}\beta_1+\cdots +v_{a(n-1)}\beta_{n-1}+v_{an} \beta_n
\end{equation}
and the $\mathcal H_a, 1\le a\le m$ and $\beta_i, 1\le i\le n$ form a basis of $\pi_2(\overline X,L)$.

If $L$ is of Chekanov type, there is no canonical isomorphism $\pi_1(L)\cong N$ or $\pi_2(\overline X, L)\cong \mathbb Z^{n+m}$, as the monodromy issue occurs.
To fix notations, we choose a path $\sigma:\oi\to B_0$ that passes through the wall component $H_n$ with $\sigma(0)\in B_+$ and $\sigma(1)\in B_-$. (Taking a different wall component $H_{n'}$ with $n'\neq n$ will yield different but almost equivalent results.) Clearly, the path induces an isomorphism $\mathscr P: \pi_2(\overline X, L_{\sigma(0)})\cong \pi_2(\overline X, L_{\sigma(1)})$.
All the holomorphic disks that is bounded by $L$ are contained in a rank-one subgroup of $\pi_2(\overline X, L)$ (Lemma \ref{holo_disk_B_-_lem}), and we denote by $\hat\beta$ the preferred generator.
See Figure \ref{figure_B_-_class}.
Note that $\hat\beta\cdot D_i=0$, $\hat\beta\cdot D_k'=0$, and $\hat\beta\cdot \mathscr E=1$.
Moreover, $\mathscr P(\beta_n) = \hat\beta$ (Lemma \ref{pi2_monodromy_lem}), and we write $\gamma_k:= \mathscr P (\beta_k-\beta_n)$ for $k=1,2,\dots, n-1$.
Eventually, abusing the notations, we would rather write:
\begin{equation}
\label{wall_cross_beta_Cn_eq}
\beta_n=\hat\beta, \qquad \gamma_k :=\beta_k-\hat\beta \quad (1\le k\le n-1)
\end{equation}
Combining (\ref{H_a_eq}) with (\ref{wall_cross_beta_Cn_eq}) yields that
\begin{equation}
\label{beta'_a_express_eq}
\beta'_a= \mathcal H_a- p_a \hat\beta -\gamma(a)
\end{equation}
where we set $p_a:=\sum_{i=1}^n v_{ai}\in\mathbb  Z$ and $\gamma(a) :=\sum_{k=1}^{n-1} v_{ak}\gamma_k \in\pi_2(\overline X,L)$ has Maslov index zero.

\subsection{Family Floer mirror: a brief outline}
\label{ss_Family_Floer_outline}

\subsubsection{General aspects}
\label{sss_family_Floer_general}
The main theorem in \cite{Yuan_I_FamilyFloer} implements the non-archimedean SYZ mirror construction with full quantum corrections for a \textit{semipositive} smooth Lagrangian fibration $\pi:X_0\to B_0$ in an ambient symplectic manifold $X$. The semipositive condition says that any stable disk in $X$ that is bounded by a fiber of $\pi$ has a nonnegative Maslov index.
We emphasize that the construction depends on the ambient space $X$. In fact, although the Lagrangian fibers are smooth and contained in $X_0$, the holomorphic disks sweep in $X$ and often meet the singular fibers.

The outcome of the mirror construction is a triple $\mathbb X^\vee=(X^\vee, W^\vee, \pi^\vee)$ consisting of a rigid analytic space $X^\vee$, a globally-defined superpotential function $W^\vee$, and a dual fibration map $\pi^\vee: X^\vee \to B_0$.
Instead of describing the subtle construction in detail, we just indicate several features and properties below, which should be sufficient for our purpose. (We also give a review in Appendix \ref{S_nonarchimedean_review}.)

The mirror rigid analytic space $X^\vee$ is first set-theoretically $\bigsqcup_{q\in B_0} H^1(L_q; U_\Lambda)$, where
$
U_\Lambda
$
is the Novikov unitary group which consists of all norm-one elements in the Novikov field $\Lambda$.
Its rigid analytic space structure is determined as follows: The dual map $\pi^\vee$ is locally modeled on the map $\trop=\val^n :(\Lambda^*)^n\to \mathbb R^n$, where $\val$ is the valuation map on the Novikov field $\Lambda=\mathbb C((T^{\mathbb R}))$ and $\Lambda^*=\Lambda\setminus\{0\}$. The two adjacent fibers of $\trop$ connect with each other via $\phi_0: y_i\mapsto T^{c_i}y_i$, while the two adjacent fibers of $\pi^\vee$ connect with each other via a gluing map in the form $\phi: y_i\mapsto T^{c_i} y_i \exp(F_i(y_1,\dots, y_n))$ for some extra formal power series $F_i$ which encodes the quantum corrections of Maslov-zero disks.
If the two corresponding Lagrangian fibers is connected by a Lagrangian isotopy that does not bound any Maslov-zero disk, then there is no wall-crossing, and the gluing map $\phi$ goes back to $\phi_0$.
Consequently, if for $U\subset B_0$ we have an integral affine chart $\psi: (U,q) \xhookrightarrow{} (\mathbb R^n, 0)$ such that the torus fibers over $U$ do not bound any nontrivial Maslov-zero holomorphic disks, then there is an isomorphism $\pi^\vee|_U\cong \trop|_{\psi(U)}$.
When $U=\psi(U)$ happens to be a rational polyhedron, the points in the preimage $\trop^{-1}(U)$ can be identified with the set of maximal ideals in an affinoid algebra in $\Lambda[[\pi_1(L_q)]]\cong \Lambda[[Y_1^\pm,\dots, Y_n^\pm]]$.
In general, a gluing map between two adjacent dual fibers over $q, \tilde q\in B_0$ can be intrinsically expresses as
\begin{equation}
	\label{gluing_intrinsic_eq}
	\phi: \Lambda[[\pi_1(L_q)]]\to \Lambda[[\pi_1(L_{\tilde q})]], \qquad Y^\alpha \to T^{\langle \alpha, \tilde q-q\rangle } Y^{\tilde \alpha} \exp \langle \tilde \alpha , \pmb F(Y)\rangle 
\end{equation}
where $\pmb F(Y)\in \Lambda[[Y_1^\pm, \dots, Y_n^\pm]]\hat\otimes H^1(L)$ is determined only by the virtual counts of Maslov-zero disks, and $\alpha\in \pi_1(L_q)$ and $\tilde \alpha \in \pi_1(L_{\tilde q})$ are matched with each other via a small isotopy.

On the other hand, let $\mathsf n_\beta$ denote the \textit{open GW invariant} of $\beta\in \pi_2(X,L_q)$ with $\mu(\beta)=2$, i.e. the counts of Maslov-two holomorphic disks (see \S \ref{ss_OGW}). We first define the local superpotentials:
\begin{equation}
\label{superpotential_intro_local_eq}
W^\vee|_q=\sum T^{E(\beta)}Y^{\partial\beta}\mathsf n_{\beta} \in \Lambda[[\pi_1(L_q)]]
\end{equation}
By the wall-crossing formula in \cite{Yuan_I_FamilyFloer}, all the local superpotentials $W^\vee|_q$ ($q\in B_0$) can be patched together via the various gluing maps in the form of (\ref{gluing_intrinsic_eq}) to produce the global superpotential $W^\vee$.

\subsubsection{Wall-crossing for the Gross's fibrations}
\label{sss_wall_cross_Gross}

In our situation, we work with the Gross's fibration $\pi:X_0\to B_0$ in (\ref{Gross_fibration_begin}).
First, we choose the ambient space to be $X=\mathbb C^n$. By \cite[Lemma 3.1]{AuTDual}, one can easily check that $\pi$ is semipositive in $X$.
So, we have the family Floer mirror triple as above, denoted by $\mathbb X^\vee= (X^\vee, W^\vee, \pi^\vee)$.
The Lagrangian fiber $L_q$ bounds a nontrivial Maslov-zero holomorphic disk if and only if $q\in H$ (Lemma \ref{Maslov_zero_locate_lem}). So, the discussion in \S \ref {sss_family_Floer_general} tells that the wall $H$ separates the mirror analytic space $X^\vee$ into the two chambers $X_\pm^\vee:=(\pi^\vee)^{-1}(B_\pm)\cong \trop^{-1}(B_\pm)$; they are glued along $(\pi^\vee)^{-1}(H)$ by a gluing map $\phi$ in the form (\ref{gluing_intrinsic_eq}) with a nontrivial series $\pmb F(Y)$.
Abusing the terminologies, we also call $X_\pm^\vee$ the Clifford/Chekanov chamber.

Next, we choose the ambient space to be $\overline X$. Since it is Fano, the $\pi:X_0\to B_0$ is also semipositive in $\overline X$.
More importantly, we observe that the collection of Maslov-zero holomorphic disks stay the same no matter what ambient space we choose (c.f. \cite[\S 5]{AuTDual}). It follows that all the gluing maps $\phi$ in (\ref{gluing_intrinsic_eq}) keep unchanged as well.
Thus, just by definition, the new mirror triple $\overline {\mathbb X}^\vee=(X^\vee, \overline W^\vee, \pi^\vee)$ with respect to $\overline X$ differs from the previous one $\mathbb X^\vee$ merely in the superpotential; the new Maslov-two holomorphic disks after the compactification contribute to the extra monomials in $\overline W^\vee-W^\vee$.

Now, let $W_\pm^\vee$ (resp. $\overline W_\pm^\vee$) denote the restrictions of $W^\vee$ (resp. $\overline W^\vee$) over the two chambers $X_\pm^\vee$.
By \cite{Cho_Oh}, we know $\mathsf n_{\beta_i}=1$, thus, the Clifford superpotential over $B_+$ for $X=\mathbb C^n$ is given by
\begin{align}
\label{W_+_eq}
W_+^\vee= 
T^{E(\beta_1)}Y^{\partial\beta_1} + \cdots + T^{E(\beta_n)}Y^{\partial\beta_n} 
= T^{E(\hat\beta)} Y^{\partial \hat\beta}
\left(
1+ T^{E(\gamma_1)} Y^{\partial\gamma_1} +\cdots + T^{E(\gamma_{n-1})} Y^{\partial\gamma_{n-1}}
\right)
\end{align}
where we also use (\ref{wall_cross_beta_Cn_eq}).
The Novikov coefficients $T^{E(\beta_i)}$ actually depend on a base point $q$, but as discussed before, the amounts of changes for choosing another base point $q'$ within $B_+$ can be exactly made up by some map $\phi_0: y_i\mapsto T^{c_i}y_i$.
For the Chekanov superpotential, only $\mathsf n_{\hat\beta}\neq 0$ (see Lemma \ref{holo_disk_B_-_lem}). It is also proved in \cite[Lemma 4.31]{CLL12} \cite{AuTDual} that $\mathsf n_{\hat\beta}=1$. Hence,
\begin{align}
\label{W_-_eq}
W_-^\vee= T^{E(\hat\beta)} Y^{\partial\hat\beta}
\end{align}
For an independent interest, we will use the family Floer theory to give a totally different proof of the fact $\mathsf n_{\hat\beta}=1$ in Theorem \ref{wall_cross_identity_intro_thm} without explicitly finding the disk like \cite{CLL12} \cite{AuTDual}.

On the other hand, by \cite{Cho_Oh} again, we also know $\mathsf n_{\beta_a'}=1$, thus, the new Clifford superpotential is
\begin{equation}
\label{W_+_overline_eq}
\overline W_+^\vee =W_+^\vee + \sum_{a=1}^m T^{E(\beta'_a)} Y^{\partial\beta'_a}
\end{equation}
However, as far as we know, the superpotential $\overline W_-^\vee$ (and the related open GW invariants) of the Chekanov tori for most examples are unknown.
A partial list of known examples is as follows: 

\begin{itemize}
	\item $\overline X=\mathbb {CP}^1\times \mathbb {CP}^1, \mathbb {CP}^2$. Then, $\overline W^\vee_-$ is found by Auroux, Chekanov-Schlenk \cite[5.7 \& 5.12]{AuTDual}.
	\item $\overline X=\mathbb {CP}^n$. Then, $\overline W^\vee_-$ over $\mathbb C$ is in Pascaleff-Tonkonog \cite[Theorem 1.4]{PT_mutation}. See Remark \ref{PT_rmk}.
\end{itemize}

The main theorem below gives a complete answer: we are able to find $\overline W^\vee_-$ for any smooth Fano toric compactification $\overline X$ (not just $\mathbb {CP}^n$) of $X=\mathbb C^n$. Especially, our computation of $\overline W_-^\vee$ is over the Novikov field $\Lambda$ rather than over $\mathbb C$. This is crucial to check a folklore conjecture for the critical values of the superpotential (c.f. \cite{Yuan_c_1_2021family, Yuan_local_SYZ_2022family}).

\begin{rmk}
	The overall line of ideas for our main result goes as follows:
First, by \cite{Cho_Oh}, we can find the Clifford superpotentials $W_+^\vee$ and $\overline W_+^\vee$ for both $X$ and $\overline X$.
Second, for $X$, one can also directly compute the Chekanov superpotential $W_-^\vee$, say, by maximal principle (c.f. \cite[Proposition 4.32]{CLL12}).
Third, one can check the Maslov-0 disks keep the same in $X$ and $\overline X$, so the general theory of \cite{Yuan_I_FamilyFloer} ensures that the two pairs $(W_-^\vee, W_+^\vee)$ and $(\overline W_-^\vee, \overline W_+^\vee)$ are related by the same gluing map. Together with the $T^{n-1}$-symmetry of the Gross's fibration, we finally deduce the Chekanov superpotential $\overline W_-^\vee$ for $\overline X$.
\end{rmk}

\section{Main theorem and its proof}

Now, we state the main theorem of this paper with the notations introduced above.
(One may better appreciate the theorem, putting it together with the examples in \S \ref{S_example}, notably 
Remark \ref{PT_rmk}.)


\begin{thm} \label{Main_theorem_superpotential_compute_thm}
Let $\pi:X_0\to B_0$ be the Gross's fibration in (\ref{Gross_fibration_begin}). For the smooth toric Fano variety $\overline X$, the superpotential function over the Chekanov chamber $B_-$ is given by
\begin{align*}
	\overline W_-^\vee 
	=
	T^{E(\hat\beta)}Y^{\partial\hat\beta} + \sum_{a=1}^m T^{E(\beta'_a)} Y^{-p_a \partial\hat\beta-\partial\gamma(a)}
	\Big( 1+ \sum_{k=1}^{n-1} T^{E(\gamma_k)} Y^{\partial\gamma_k}
	\Big)^{p_a}
\end{align*}
Hence, we obtain all the one-pointed open GW invariants for a Chekanov-type Lagrangian torus $L_-$.
Specifically, for any Maslov index two topological class $\beta\in \pi_2(\overline X,L_-)$, we know whether or not there exists a holomorphic disk $u$ with $[u]=\beta$ as well as the number of its virtual counts if exists.
\end{thm}

\begin{proof}
Take a gluing map $\phi$ along $(\pi^\vee)^{-1}(H_n)$ associated to a path $\sigma$ in $B_0$ across $H_n$. Without loss of generalities, we may identify $\pi_2(\overline X, L_{\sigma(0)})\cong \pi_2(\overline X, L_{\sigma(1)})$ and $\pi_1(L_{\sigma(0)})\cong \pi_1(L_{\sigma(1)})$; we may also omit the Novikov field coefficient $T^{\langle \alpha, \tilde q-q\rangle}$ in (\ref{gluing_intrinsic_eq}) thanks to the Fukaya's trick. Namely, we may assume the gluing map takes the form $\phi: Y^\alpha \mapsto Y^\alpha \exp \langle \alpha, \pmb F\rangle$ where $\alpha\in \pi_1(L)$.
Since $\pmb F(Y)$ is a $H^1(L)$-valued formal series determined by the Maslov-zero disk counting, the compactification does not affect $\phi$ (see \S \ref{sss_wall_cross_Gross}). Thus, the wall crossing formula in the family Floer program \cite{Yuan_I_FamilyFloer} implies that
\begin{equation}
\label{phi_X_eq}
\phi(W^\vee_+)=W^\vee_-
\end{equation}
\begin{equation}
\label{phi_X_overline_eq}
\phi(\overline W^\vee_+)=\overline W^\vee_-
\end{equation}

Now, we claim that
\begin{equation}
	\label{F_gamma_eq}
	\langle \partial\gamma_k, \pmb F\rangle=0 \qquad   \text{for} \ k=1,2,\dots,n-1
\end{equation}

	It is essentially due to the second paragraph of the proof of \cite[Theorem 8.4]{AAK_blowup_toric}, and we only give a sketch here.
	A Gross's fiber can be decomposed into an $S^1$-family of $T^{n-1}$-orbits (\S \ref{sss_Gross_fibration_defn}). The key geometric observation is that the boundary of a Maslov-zero holomorphic disk is always contained in a $T^{n-1}$-orbit. Indeed, let $u$ be a Maslov index zero holomorphic disk that is bounded by a Gross's fiber $L\cong T^n$. The class $\gamma:=[u]$ is a linear combination of $\gamma_1,\dots,\gamma_{n-1}$;
	the boundary $\partial u$ is contained in $T':=L\cap \bigcup_i D_i \cong T^{n-1}$ (Corollary \ref{Maslov_zero_bdry_cor}), and so the evaluation map $\ev: \mathcal M_{1,\gamma}(L)\to L$ is supported in the sub-torus $T'$.
	Now, up to a scalar, a monomial in $\pmb F$ is a class in $H^1(L)\cong H_{n-1}(L)$ defined by the counts of Maslov-zero disks. Namely, it is given by the pushforward of the evaluation map on the moduli spaces and is therefore dual to $T'$ in $L\cong T^n$.
	Finally, since $\pi_1(T')=H_1(T')\cong \mathbb Z^{n-1}$ is exactly generated by $\partial\gamma_1,\dots,\partial\gamma_{n-1}$ (Lemma \ref{s_q_theta_describe_lem}), we have $\langle \partial\gamma_k,\pmb F\rangle=0$.

Notice that
$\exp \langle \partial\hat\beta, \pmb F\rangle  \big(
1+ T^{E(\gamma_1)} Y^{\partial\gamma_1} \exp \langle \partial\gamma_1, \pmb F\rangle +\cdots + T^{E(\gamma_{n-1})} Y^{\partial\gamma_{n-1}} \exp \langle \partial\gamma_{n-1}, \pmb F\rangle
\big)=1$ by virtue of (\ref{W_+_eq}) (\ref{W_-_eq}) (\ref{phi_X_eq}). Now, the gluing map $\phi$ is determined, since it follows from (\ref{F_gamma_eq}) that
\begin{equation}
\label{C2_identity_eq}
\exp \langle -\partial\hat\beta, \pmb F\rangle  = 
1+ T^{E(\gamma_1)} Y^{\partial\gamma_1} +\cdots +T^{E(\gamma_{n-1})} Y^{\partial\gamma_{n-1}}
\end{equation}
As $\overline W_+^\vee$ is known (\ref{W_+_overline_eq}), one can further use (\ref{phi_X_overline_eq}) to compute $\overline W_-^\vee$ directly. The proof is complete.
\end{proof}

\begin{rmk}
First, knowing the gluing map $\phi$ takes the form (\ref{gluing_intrinsic_eq}) is extremely crucial for the proof.
Second, the two indispensable relations (\ref{phi_X_eq}, \ref{phi_X_overline_eq}) require the wall crossing formula in \cite{Yuan_I_FamilyFloer}, which is highly non-trivial to prove: we must perform careful studies of both the non-archimedean analysis and the homological algebra of $A_\infty$ structures.
Hence, although the proof of Theorem \ref{Main_theorem_superpotential_compute_thm} looks very brief, the foundation work \cite{Yuan_I_FamilyFloer} is really complicated.
In short, we simply apply an involved new mechanism (\S \ref{S_nonarchimedean_review}) to a concrete classic example (\S \ref{S_Gross_fibration}), producing a surprisingly concise proof.
\end{rmk}

\begin{rmk}
	\label{shrink_B_0_to_path}
	Notice that we do not have to use a full fibration for the proof, as one flexibility of our story is that the total space $X_0$ of the torus fibration $\pi$ can be actually chosen small in $X$ or $\overline X$. Even, one can simply start with a single Lagrangian isotopy (imagine the base $B_0$ shrinks to a path).
	We hope this perspective may be useful for the further studies.
\end{rmk}

\section{Examples}
\label{S_example}

We select several examples of Theorem \ref{Main_theorem_superpotential_compute_thm} as follows. There will be definitely more other examples.
Furthermore, in principle, if we replace $X=\mathbb C^n$ by another suitable toric Calabi-Yau manifold and consider a similar compactification, the method should also obtain some other results.

\subsubsection{\textbf{Projective spaces}} Take $m=1$ and $v:=v_1=e_1+\cdots+ e_n$. Then, $\overline X=\mathbb {CP}^n$. 
Note that $\beta':=\beta'_1$ satisfies $\partial\beta'=v=e_1+\cdots+e_n$. So, $\mathcal H:=\mathcal H_1=\beta'+\beta_1+\cdots+\beta_n$.
Besides, by (\ref{beta'_a_express_eq}), we have $p_1=n$ and $\beta'=\mathcal H -n\hat\beta- (\gamma_1+\cdots+\gamma_{n-1})$.
Hence, due to Theorem \ref{Main_theorem_superpotential_compute_thm}, the superpotential for a Chekanov-type torus $L$ in the ambient space $\overline X=\mathbb {CP}^n$ is given by
\begin{equation}
\label{CPn_W_-_eq}
\overline W_-^\vee= T^{E(\hat\beta)} Y^{\partial\hat\beta} + T^{E(\beta')} Y^{-n\partial \hat\beta- (\partial\gamma_1+\cdots+\partial\gamma_{n-1})} \big(1+ T^{E(\gamma_1)}Y^{\partial\gamma_1}+\cdots + T^{E(\gamma_{n-1})} Y^{\partial\gamma_{n-1}}
\big)^n
\end{equation}
On the other hand, by its geometric definition, we also know that
\[
\textstyle
\overline W_-^\vee = \sum_{\beta\in\pi_2(X,L)} T^{E(\beta)} Y^{\partial\beta} \mathsf n_\beta
\]
where the summation runs over all the Maslov index two classes $\beta\in\pi_2(\overline X,L)$.
Thus, we can compute $\mathsf n_\beta$ by simply comparing the coefficients of the above two equations. Then, it is direct to obtain:
\begin{cor}
	\label{nonzero_openGW_CPn_cor}
	For a Chekanov-type torus in $\mathbb {CP}^n$, the open GW invariant $\mathsf n_\beta\neq 0$ if and only if $\beta=\hat\beta$ or
	\begin{equation}
		\label{CP_n_topo_class_eq}
	\beta=\mathcal H-n\hat\beta +k_1\gamma_1+\cdots + k_{n-1}\gamma_{n-1}
	\qquad  \forall \
	k_1,\dots, k_{n-1}\ge -1, \
	k_1+\cdots+k_{n-1}\le 1
	\end{equation}
	Moreover, these open GW invariants are explicitly given by $\mathsf n_{\hat\beta}=1$ and
	\[
	\mathsf n_{\mathcal H-n\hat\beta + k_1 \gamma_1 +\cdots+ k_{n-1}\gamma_{n-1}} = \frac{n!}{(k_1+1)!\cdots (k_{n-1}+1)! (1-k_1-\cdots-k_{n-1})!} 
	\]
\end{cor}

\begin{rmk}
	If $n=2$, then $\overline X=\mathbb CP^2$. We have $k:=k_1\in \{-1,0,1\}$ and $\gamma:=\gamma_1$. By Corollary \ref{nonzero_openGW_CPn_cor}, all the nontrivial open GW invariants are $\mathsf n_{\hat\beta}=1$, $\mathsf n_{\mathcal H-2\hat\beta-\gamma}=1$, $\mathsf n_{\mathcal H-2\hat\beta+\gamma}=1$, and $\mathsf n_{\mathcal H-2\hat\beta}=2$.
	The outcomes exactly agree with the Chekanov-Schlenk's result; see \cite[5.6 \& 5.7]{AuTDual}.
\end{rmk}

\begin{rmk} \label{PT_rmk}
	As mentioned above, Pascaleff-Tonkonog \cite[Theorem 1.4]{PT_mutation} have also obtained the superpotential (\ref{CPn_W_-_eq}) using a different approach over the complex field $\mathbb C$. Specifically, for each $1\le k\le n$, they show that $\mathbb{CP}^n$ contains a monotone Lagrangian torus whose superpotential is
	$W_{PT,k}= \sum_{i=k}^n x_i^{-1} +x_k^k \cdot \big(\sum_{i=1}^k x_i^{-1} \big)^k\prod_{i=1}^n x_i$.
	Notably, when $k=n$, we can rearrange it as follows:
	\[
	W_{PT,n}= x_n^{-1} + x_n^n\cdot (x_nx_1^{-1})^{-1}\cdots (x_n x_{n-1}^{-1})^{-1} \left(1+ x_nx_1^{-1}+\cdots+ x_n x_{n-1}^{-1}\right)^n
	\]
	Then, surprisingly, the change of variables
	$x_n^{-1} \xleftrightarrow{} Y^{\partial\hat\beta}$ and
$x_nx_k^{-1} \xleftrightarrow{} Y^{\partial\gamma_k}$ ($1\le k\le n-1$)
	will almost match their $W_{PT,n}$ with ours in (\ref{CPn_W_-_eq}).
	As we see, in the case of $\mathbb{CP}^n$, Pascaleff-Tonkonog's work is definitely more general, including more classes of Lagrangian tori apart from the Clifford/Chekanov ones.
	Nevertheless, a mild advantage of our result is that we can also explicitly find the corresponding topological classes in $\pi_2(\mathbb {CP}^n, L)$ for the nontrivial open GW invariants as listed in (\ref{CP_n_topo_class_eq}).

In view of Remark \ref{shrink_B_0_to_path}, it is also expected that our method might obtain their other superpotentials $W_{PT,k}$ with $k\neq n$.
On the other hand, the coincidence of the two different approaches strongly implies an underlying interconnection. Pascaleff-Tonkonog \cite{PT_mutation} mentioned that a direct study of Maslov-zero holomorphic disks is avoided and bypassed by using the Seidel's ideas in \cite{Seidel_dynamics}.
	We believe that there is an intrinsic reason for this coincidence, but we do not have a good answer at this moment.
\end{rmk}

\subsubsection{\textbf{Product of projective spaces}}
Fix $1\le r\le n$. Take $m=2$, $v_1=e_1+\cdots+e_r$, and $v_2=e_{r+1}+\cdots +e_n$. Then $\overline X=\mathbb {CP}^r\times \mathbb {CP}^{n-r}$.
Notice that $\beta'_1,\beta'_2$ satisfy $\partial\beta'_1=v_1$ and $\partial\beta'_2=v_2$. Thus, $\mathcal H_1=\beta'_1+\beta_1+\cdots+\beta_r$ and $\mathcal H_2=\beta'_2+\beta_{r+1}+\cdots+ \beta_n$.
Moreover, due to (\ref{beta'_a_express_eq}), we have $p_1=r$, $p_2=n-r$, and $\beta'_1=\mathcal H_1-r\hat\beta -\gamma(1)$, $\beta'_2=\mathcal H_2-(n-r)\hat\beta -\gamma(2)$ where $\gamma(1)=\gamma_1+\cdots+\gamma_r$ and $\gamma(2)=\gamma_{r+1}+\cdots+\gamma_n$. Therefore, by Theorem \ref{Main_theorem_superpotential_compute_thm}, the superpotential for a Chekanov-type torus $L$ in the ambient space $\overline X=\mathbb {CP}^r\times \mathbb {CP}^{n-r}$ is given by
\[
\overline W_-^\vee = T^{E(\hat\beta)} Y^{\partial\hat\beta} + T^{E(\beta'_1)} Y^{-r\partial\hat\beta - \partial\gamma(1)} \Big(1+\sum_{j=1}^r T^{E(\gamma_j)}Y^{\partial\gamma_j} \Big)^r
+
T^{E(\beta'_2)} Y^{-(n-r)\partial\hat \beta- \partial \gamma(2)}
\Big(1+\sum_{j=r+1}^n T^{E(\gamma_j)} Y^{\partial \gamma_j}\Big)^{n-r}
\]
Similarly, by comparison again, a straightforward calculation yields that

\begin{cor}
	\label{nonzero_openGW_CPrCPr_cor}
	For a Chekanov-type torus in $\mathbb {CP}^r \times \mathbb {CP}^{n-r}$, the open GW invariant $\mathsf n_\beta\neq 0$ if and only if $\beta=\hat\beta$ or
	\[
	\beta=\mathcal H_1-r\hat\beta +k_1\gamma_1+\cdots+ k_{n-1}\gamma_{n-1}
	\quad \forall \
	k_1,\dots, k_r \ge -1, \ \
	k_{r+1},\dots, k_{n-1}\ge 0, \ \
	k_1+\cdots +k_{n-1}\le 0
	\]
	or
	\[
	\beta=\mathcal H_2-(n-r)\hat\beta +\ell_1\gamma_1+\cdots+ \ell_{n-1} \gamma_{n-1} \quad \forall \
	\ell_1,\dots, \ell_r \ge 0, \ \
	\ell_{r+1},\dots, \ell_{n-1} \ge -1, \ \
	\ell_1 + \cdots +\ell_{n-1}\le 1
	\]
	Furthermore, these nontrivial open GW invariants are concretely given by $\mathsf n_{\hat\beta}=1$ and
	\[
	\mathsf n_{\mathcal H_1 - r\hat\beta
		+k_1\gamma_1+\cdots + k_{n-1}\gamma_{n-1}}
	=
	\frac{r!}{(k_1+1)!\cdots (k_r+1)! k_{r+1}!\dots k_{n-1}! (-k_1-\cdots-k_{n-1})!}
	\]
	and
	\[
	\mathsf n_{\mathcal H_2 -(n-r)\hat\beta  +\ell_1 \gamma_1+\cdots + \ell_{n-1} \gamma_{n-1}}
	=
	\frac{(n-r)!}{\ell_1!\cdots \ell_r! (\ell_{r+1}+1)!\cdots (\ell_{n-1}+1)! (1-\ell_1-\cdots -\ell_{n-1})!}
	\]
\end{cor}

\begin{rmk}\label{retrieve_CP1CP1_rmk}
	Suppose $n=2$ and $r=1$. Then, we only have $k:=k_1$ and $\ell:=\ell_1$ which are subject to the conditions $-1\le k\le 0$ and $0\le \ell \le 1$. 
	Set $\gamma=\gamma_1$.
	The classes of Maslov index two holomorphic disks are $\hat\beta$, $\mathcal H_1-\hat\beta-\gamma$, $\mathcal H_1-\hat\beta$, $\mathcal H_2-\hat\beta$, and $\mathcal H_2-\hat\beta +\gamma$. By Corollary \ref{nonzero_openGW_CPrCPr_cor}, their open Gromov-Witten invariants are all equal to $1$. Hence, we exactly retrieve \cite[Proposition 5.12]{AuTDual}.
\end{rmk}

\subsubsection{\textbf{Hirzebruch surface}}
Assume $n=2$.
Take $m=2$, $v_1=e_1+e_2$, and $v_2=e_2$.
Then $\overline X=\mathbb P \big(\mathcal O_{\mathbb P^1}\oplus \mathcal O_{\mathbb P^1}(1) \big)=:\mathbb F_1$ is a Hirzebruch surface. It is also the del Pezzo surface given by the blow-up of $\mathbb {CP}^2$ at one of the torus fixed points.

As before, we have $\beta'_1$ and $\beta'_2$ such that $\partial\beta'_1=v_1=e_1+e_2$ and $\partial\beta'_2=v_2=e_2$.
Then,
$\mathcal H_1=\beta'_1+\beta_1+\beta_2$ and $\mathcal H_2=\beta'_2+\beta_2$. 
So, by (\ref{W_+_overline_eq}), the Clifford superpotential is given by
\[
\overline W_+^\vee= T^{E(\beta_1)}Y^{\partial\beta_1} + T^{E(\beta_2)} Y^{\partial\beta_2} + T^{E(\mathcal H_1-\beta_1-\beta_2)} Y^{-\partial\beta_1-\partial\beta_2} + T^{E(\mathcal H_2-\beta_2)} Y^{-\partial\beta_2}
\]
On the other hand, we set $\gamma=\gamma_1$.
Due to (\ref{wall_cross_beta_Cn_eq}) (\ref{beta'_a_express_eq}), we further have $\beta'_1=\mathcal H_1 - 2\hat\beta -\gamma$ and $\beta'_2=\mathcal H_2-\hat\beta$.
Then, it follows from Theorem \ref{Main_theorem_superpotential_compute_thm} that the Chekanov superpotential is
\begin{align*}
\overline W_-^\vee
&
=
T^{E(\hat\beta)}Y^{\partial\hat\beta}
+
T^{E(\beta'_1)}Y^{-2\partial\hat\beta-\partial\gamma}
(1+ T^{E(\gamma)} Y^{\partial\gamma} )^2
+
T^{E(\beta'_2)} Y^{-\partial\hat\beta}	\\
&
=
T^{E(\beta)} Y^{\partial\hat\beta} 
+
T^{E(\mathcal H_1-2\partial\hat\beta-\partial\gamma)} Y^{-2\partial\hat\beta -\partial\gamma}
+
2 T^{E(\mathcal H_1-2\partial\hat\beta)} Y^{-2\partial\hat\beta}
+
T^{E(\mathcal H_1-2\partial\hat\beta+\partial\gamma)} Y^{-2\partial\hat\beta +\partial\gamma}
+
T^{E(\mathcal H_2-\partial\hat\beta)} Y^{-\partial\hat\beta}
\end{align*}
Extracting the coefficients, we analogously obtain the following result (which seems to be new):

\begin{cor}
	For a Chekanov-type torus in $\mathbb F_1$, the open GW invariant $\mathsf n_\beta\neq 0$ if and only if $\beta=\hat\beta, \mathcal H_1-2\hat\beta-\gamma, \mathcal H_1-2\hat\beta, \mathcal H_1-2\hat\beta+\gamma,$ or $\mathcal H_2-\hat\beta$.
	The corresponding values of open GW invariants are $1,1,2,1,1$.
\end{cor}

\paragraph{\textbf{Acknowledgements}.}
I am indebted to my advisor Kenji Fukaya for many enlightening discussions and conversations.
I would like to thank Yuhan Sun for his knowledge of toric geometry and enlightening discussions.
I would also like to thank Mohammed Abouzaid, Jiahao Hu, Dogancan Karabas, Wenyuan Li, Santai Qu, Renato Vianna, Junxiao Wang, Yi Wang, and Eric Zaslow for helpful conversations.
I am grateful to Siu Cheong Lau and Yu-Shen Lin for the invitation to Boston University Geometry and Physics Seminar in Fall 2020.
I am also grateful to Sara Tukachinsky for the invitation to Symplectic Zoominar in Spring 2021.

\appendix

\section{Non-archimedean SYZ mirror construction: Review}
\label{S_nonarchimedean_review}

In this section, we review the construction of $\mathbb X^\vee=(X^\vee, \pi^\vee, W^\vee)$. See \cite{Yuan_I_FamilyFloer} for the full details.

\subsection{Preparation}

\subsubsection{Non-archimedean analysis}
We will need some non-archimedean analysis. 
Recall that the \textit{Novikov field} $\Lambda$ consists of all the series $\sum_{i\ge 0} a_i T^{E_i}$ where $T$ is a formal symbol, $a_i\in\mathbb C$, and $\{E_i\}$ is a divergent strictly-increasing sequence of real numbers. It comes with a valuation map $\val: \Lambda\to \mathbb R\cup\{\infty\}$ defined by setting $0$ to $\infty$ and sending $\sum_{i\ge 0} a_i T^{E_i}$ (with all $a_i\neq 0$) to the smallest $E_0$.
The $\val$ is equivalent to a non-archimedean norm $|a|=\exp(-\val(a))$, giving an adic-topology on $\Lambda$.
Next, we consider the \textit{non-archimedean torus fibration}:
\[
\trop: (\Lambda^*)^n \to \mathbb R^n, \quad (z_i)\to (\val(z_i))
\]
which is an analog of the complex logarithm map
$Log: (\mathbb C^*)^n\to\mathbb R^n, \quad (z_i)\mapsto (\log|z_i|)$.
The preimage $
\trop^{-1}(\Delta)
$ of a rational polyhedron $\Delta\subset \mathbb R^n$ is an \textit{affinoid domain} (an analogue of affine scheme in the theory of rigid analytic spaces), and we call it a \textit{polytopal domain}.
We remark that the points in
$\trop^{-1}(\Delta)$ are in bijection with the maximal ideals of the \textit{polyhedral affinoid algebra}:
\begin{equation}
	\label{polyhedral_affinoid_algebra_1_eq}
	\textstyle \Lambda\langle \Delta\rangle:= \Big\{ \sum_{\nu_i\in\mathbb z} a_{\nu_1\cdots \nu_n} Y_1^{\nu_1}\cdots Y_n^{\nu_n} \ \ \Big|\ \ \val(a_{\nu_1\cdots \nu_n}) +  (\nu_1,\dots, \nu_n)\cdot u \to \infty \ \forall \ u\in \Delta 
	\Big\}
\end{equation}
Alternatively, the algebra $\Lambda\langle\Delta\rangle$ consists of all Laurent formal power series $f\in \Lambda[[Y_1^\pm,\dots,Y_n^\pm]]$ so that for every point $(y_1,\dots,y_n)\in \trop^{-1}(\Delta)$, the $f(y_1,\dots,y_n)$ forms a convergent series with respect to the aforementioned adic-topology in $\Lambda$.

The mirror fibration $(X^\vee,\pi^\vee)$ over a sufficiently small domain $\Delta\subset B_0$ is isomorphic to $\trop^{-1}(\Delta)$.
Indeed, the $B_0$ is an integral affine manifold, we can similarly define a rational polyhedron $\Delta$ in $B_0$:
given a rational polyhedron $\Delta$ in $B_0$, we can find an integral affine chart $\varphi_q$ centered at some point $q\in B_0$;
then, $\varphi_q(\Delta)$ is a rational polyhedron in $\mathbb R^n$, and we have the polyhedral affinoid algebra $\Lambda\langle \varphi_q(\Delta)\rangle$ and the corresponding polytopal domain as before.
Intrinsically, we define
\begin{equation}
	\label{polyhedral_affinoid_algebra_2_eq}
	\textstyle
	\Lambda\langle \Delta, q \rangle:= \Big\{ \sum_{\alpha\in \pi_1(L_q)} a_{\alpha} Y^\alpha \ \ \Big|\ \ \val(a_{\alpha}) +  \langle  \alpha, q'-q\rangle \to \infty \ \forall \ q'\in \Delta 
	\Big\}
\end{equation}
in the formal power series ring $\Lambda[[\pi_1(L_q)]]$.
Here the $Y$ is a formal symbol, and we can view $q'-q$ as a vector in $T_q B_0\cong H^1(L_q)\cong \mathbb R^n$.
Abusing the terminologies, we also call $\Lambda\langle \Delta,q\rangle$ a \textit{polyhedral affinoid algebra}; its spectrum of maximal ideals, denoted by
$\Sp \Lambda \langle \Delta, q\rangle$, is also called a \textit{polytopal domain}.
We can similarly define $\Sp \Lambda\langle \Delta,q\rangle \to \Delta$ which is naturally isomorphic to the $\trop$ over $\varphi_q(\Delta)$ via $\varphi_q$.
For clarity, we will not distinguish between $\trop$ and $\trop_q$, $\Delta$ and $\varphi_q(\Delta)$.

\subsubsection{Homological algebra}
We start with a brief review of the homological algebra used in \cite{Yuan_I_FamilyFloer}.
A \textit{label group} is an abelian  group $(\G,E,\mu)$ with two group homomorphisms $E:\G\to\mathbb R$ and $\mu:\G\to 2\mathbb Z$.
In practice, we take a symplectic manifold $X$ and a compact oriented relatively-spin Lagrangian submanifold $L\subset X$; then, we set
$\G=\G(X,L):=\im (\pi_2(X,L)\to H_2(X,L))$
with $E$ and $\mu$ being the energy/area and the Maslov index respectively. We will not distinguish $\G(X,L)$ and $\pi_2(X,L)$.

By a \textit{$\G$-gapped $A_\infty$ algebra} on a vector space $C$, we mean an operator system $\m=(\m_{k,\beta})$, where the $\m_{k,\beta}:C^{\otimes k}\to C$ is a $k$-multilinear operator labeled by $\beta\in \G$, satisfying $\deg \m_{k,\beta}=2-k-\mu(\beta)$, the gappedness condition (which encodes Gromov's compactification), and the obvious $A_\infty$ associativity relation.
For simplicity, we will often omit saying `$\G$-gapped'.
In general, $\m_1:=\sum T^{E(\beta)}\m_{1,\beta}$ does not satisfy $\m_1\circ\m_1=0$, but in the energy-zero level, we do have $\m_{1,0}\circ\m_{1,0}=0$ by definition.
Thus, an $A_\infty$ algebra $(C,\m)$ induces the underlying cochain complex $(C,\m_{1,0})$.
We call $(C,\m)$ \textit{minimal} if $\m_{1,0}=0$.
A \textit{$\G$-gapped $A_\infty$ homomorphism} $\f=(\f_{k,\beta}): (C,\m)\to (C',\m')$ can be defined in the similar way, and it also induces an underlying cochain map $\f_{1,0}: (C,\m_{1,0})\to (C',\m'_{1,0})$.
By an \textit{$A_\infty$ homotopy equivalence} we mean an $A_\infty$ homomorphism so that $\f_{1,0}$ is a quasi-isomorphism of the underlying cochain complexes.
Due to Whitehead theorem, any $A_\infty$ homotopy equivalence admits a homotopy inverse $\g$ in the sense that $\g\circ \f$ and $\f\circ\g$ are homotopic to the identity $A_\infty$ homomorphism $\id$.
Technically, we need to include the extra unitality and divisor axiom conditions for the homotopy.

\subsubsection{$A_\infty$ algebras associated to Lagrangian submanifolds}

For an $\omega$-tame almost complex structure $J$, the (compactified) moduli space 
$\mathcal M_{k+1,\beta}(J,L)$
consists of equivalence classes of $(k+1)$-boundary-marked $J$-holomorphic stable maps of genus zero with one boundary component in the class $\beta\in \pi_2(X,L)$.
Note that the virtual dimension is equal to $n+k-2+\mu(\beta)$.
A stable map in this moduli space is often called a \textit{stable disk}.
There exists a stable map topology on the moduli space which is Hausdorff and compact \cite[Theorem 7.1.43]{FOOOBookTwo}.
Its interior part $\mathcal M^\circ_{k+1,\beta}(J,L)$ is the subset of stable disks whose source is a smooth disk (without any sphere or disk bubbles).

In practice, let $\pi:X_0\to B_0$ be a semipositive Lagrangian torus fibration.
Given $q\in B_0$, we take the moduli spaces with the Lagrangian torus fiber $L_q:=\pi^{-1}(q)$ in place of a general $L$ there.
Then, we first obtain an $A_\infty$ algebra
$
(\Omega^*(L_q), \check\m)
$
on the de Rham complex.
Exploiting the homological perturbation produces a minimal $A_\infty$ algebra on the de Rham cohomology $(H^*(L_q), \m)$.
(Technically, we need to use the so-called harmonic contraction rather than an arbitrary contraction; see \cite{Yuan_I_FamilyFloer} if interested). We will only work with $\m$ regardless of $\check \m$.
We define the \textit{Maurer-Cartan formal power series} of $\m$:
\[
\textstyle
P=\sum_\beta T^{E(\beta)} Y^{\partial\beta} \m_{0,\beta} \qquad \in \Lambda[[\pi_1(L_q)]] \hat\otimes H^*(L_q)
\]
Note that one can identify $\Lambda[[\pi_1(L_q)]]$ with $\Lambda[[Y_1^\pm,\dots, Y_n^\pm]]$ by specifying a basis.
Denote by $\one \in H^0(L_q)$ the constant-one function.
By the semipositive assumption (\S \ref{sss_family_Floer_general}), only those $\beta$'s with $\mu(\beta)=0$ or $2$ have contributions to $P$.
Then, we make the decomposition
$
P=W\cdot \one + Q$
where
\[
\textstyle
W:=\sum_{\mu(\beta)=2}  T^{E(\beta)} Y^{\partial\beta} \m_{0,\beta} / \one \qquad \quad Q:=\sum_{\mu(\beta)=0} T^{E(\beta)} Y^{\partial\beta} \m_{0,\beta}
\]
By degree reason, $\m_{0,\beta}\in H^{2-\mu(\beta)}(L)$. Remark that the number $\m_{0,\beta} / \one $ is also known as the (one-pointed) \textit{open Gromov-Witten invariant} (c.f. \S \ref{ss_OGW}).
Given a basis $\{\theta_1,\dots, \theta_n\}$ of $H^1(L_q)$, we obtain a basis $\theta_{pq}:=\theta_p\wedge \theta_q$ ($1\le p<q\le n$) of $H^2(L_q)$, and we write $Q=\sum_{p<q} Q_{pq}\cdot \theta_{pq}$.

\subsection{Construction}

\subsubsection{Mirror local charts}
By Groman-Solomon's reverse isoperimetric inequalities \cite{ReverseI},
one can prove that for a sufficiently small rational polyhedron $\Delta$, the series $W$ and $Q_{pq}$ are all contained in the polyhedral affinoid algebra $\Lambda \langle \Delta, q\rangle$.
To avoid too many digressions, we assume all the $Q_{pq}$ are zero, i.e. the weak Maurer-Cartan equations vanish. This is the case at least for the examples in the current paper, although we do not know so in general.

Now, we first remark that the mirror space $X^\vee$ is set-theoretically $\bigcup_{q\in B_0} H^1(L_q; U_\Lambda)\equiv B_0\times U_\Lambda^n$. In contrast, the total space $(\Lambda^*)^n$ of $\trop$ is set-theoretically $\mathbb R^n\times U_\Lambda^n$.
Then, a \textit{local chart} of $X^\vee$ is defined by
$\Sp\Big( \Lambda\langle \Delta,q\rangle \Big)\cong \trop_q^{-1}(\Delta)$.
The remaining series $W=\sum T^{E(\beta)} Y^{\partial\beta} \m_{0,\beta}/ \one$ will give a local piece of the global potential $W^\vee$, and the dual map $\pi^\vee$ is locally identified with the map $\trop_q$.

\subsubsection{Gluing map}
\label{sss_gluing_map}
Next, we aim to define the \textit{gluing maps} (also call transition maps) among the dual fibers.
Note that the different fibers of $\trop$ are related by the shifting maps $y_i\mapsto T^{c_i}y_i$, while the dual fibers of $\pi^\vee$ are related by the gluing maps in the form of $y_i\mapsto T^{c_i} y_i \exp(F_i(y))$ with the additional twisting term $F_i$ that is determined by the counts of Maslov-zero holomorphic disks.

Let $L=L_q$ and $\tilde L=L_{\tilde q}$ be two adjacent Lagrangian fibers.
Suppose the two associated $A_\infty$ algebras are $(H^*(L), \m)$ and $(H^*(\tilde L), \tilde \m)$.
Roughly, chosen a small diffeomorphism $F\in\diff_0(X)$ with $F(L)=\tilde L$, the first $A_\infty$ algebra is almost identical to another $A_\infty$ algebra $(H^*(\tilde L), \m^F)$ except the Novikov coefficient is varied by the relation $E(\tilde \beta)-E(\beta)=\langle \partial\beta, \tilde q-q\rangle$. Besides, it is subject to the Fukaya's trick equation:
\begin{equation}
\label{Fuk_trick_coh_eq}
\m^F_{k,\tilde \beta} = F^{-1*} \circ \m_{k,\beta} \circ (F^*,\dots, F^*)
\end{equation}
It derived from the fact that whenever $u$ is a $J$-holomorphic disk bounded by $L$, the composition map $F\circ u$ is a $F_*J$-holomorphic disk bounded by $\tilde L$, and vice versa. This observation leads to a natural identification of the moduli spaces $
\mathcal M_{k,\beta}(J,L)\cong \mathcal M_{k,\tilde \beta}(F_*J, \tilde L)$ via $u\leftrightarrow F\circ u$.

By Fukaya's trick, we unify the underlying Lagrangian of the two $A_\infty$ algebras by transferring an isotopy of Lagrangian fibers to a movement of almost complex structures; as studied in the literature \cite{FuCyclic, FOOOBookOne} the latter gives rise to an $A_\infty$ homotopy equivalence
\begin{equation}
	\label{mC_F_eq}
	\mC^F=(\mC^F_{k,\beta}) : \tilde \m \to \m^F
\end{equation}
Modulo the technical consideration of the harmonic contractions and homological algebras, it is obtained by choosing a path $\mathbf J=(J_t)_{t\in\oi}$ of $\omega$-tame almost complex structures from $J$ to $F_*J$ and by considering the following parameterized moduli space
\begin{equation}
\label{moduli_parameterized_eq}
\textstyle
\mathcal M_{k+1,\beta}(\mathbf J, L):= \bigsqcup_t  \mathcal M_{k+1,\beta}(J_t, L)
\end{equation}

Now, we consider the ring isomorphism defined as follows
\begin{equation}
\label{formula_phi_eq}
\textstyle
\phi: \Lambda[[ \pi_1(L)]] \to \Lambda[[\pi_1(\tilde L)]], \qquad
Y^\alpha \mapsto T^{\langle \alpha, \tilde q-q\rangle} \cdot Y^{F_*\alpha}\cdot \exp \Big\langle 
F_*\alpha, \sum_{\gamma\neq 0, \mu(\gamma)=0} \mC_{0,\gamma}^F T^{E(\gamma)} Y^{\partial \gamma}
\Big\rangle
\end{equation}
Note that $\mC_{0,\gamma}^F\in H^1(L)$. Besides, it further restricts to an algebra homomorphism
$
\phi: \Lambda\langle \tilde \Delta, \tilde q \rangle \to \Lambda\langle \Delta, q\rangle
$.
Then we define a {transition map} (or called a gluing map) to be the associated map on the spectrum of maximal ideals
$\psi: V=\Sp (\Lambda \langle  \Delta,   q\rangle) \to \tilde V=\Sp (\Lambda \langle \tilde\Delta, \tilde q\rangle)$.
Abusing the terminologies, we call the original homomorphism $\phi$ a \textit{gluing map} or a \textit{transition map}.

Both $\alpha\in \pi_1(L)$ and $F_*\alpha \in \pi_1(\tilde L)$ can be identified with a lattice point $\alpha=(\alpha_1,\dots,\alpha_n)$ in $\mathbb Z^n$, thus, we have $\Lambda[[\pi_1(L)]]\cong \Lambda[[\pi_1(\tilde L)]]\cong \Lambda[[Y_1^\pm,\dots, Y_n^\pm]]$ by identifying $Y^\alpha$ with $Y_1^{\alpha_1}\cdots Y_n^{\alpha_n}$.
Therefore, up to a shifting $Y_i\mapsto T^{c_i}Y_i$, we may often omit the coefficient $T^{\langle \alpha, \tilde q-q\rangle}$ in the formula of $\phi$ (\ref{formula_phi_eq}) thanks to the Fukaya's trick.
Then,
	\begin{equation}
	\label{gluing_map_Fukaya_trick_eq}
	\phi: Y^\alpha\mapsto Y^{\alpha} \exp \langle \alpha, \pmb F(Y)\rangle
	\end{equation}
	where we set
	$
	\pmb F(Y):= \sum_{\gamma\neq 0,\mu(\gamma)=0}  T^{E(\gamma)} Y^{\partial\gamma} \mC_{0,\gamma}^F$ in $\Lambda[[Y_1^\pm,\dots, Y_n^\pm]]\hat\otimes H^1(L))$.
Notice that each $\gamma$ is the class of a holomorphic stable disk, and the coefficients of $\pmb F$ are all in $\Lambda_+$.

\subsubsection{Cocycle condition}
\label{sss_cocycle}
First, note that the transition map $\psi$ is independent of the various choices, such as $F, \mathbf J,\mg$.
In brief, one can first show that another $A_\infty$ homomorphism $\mC^{F'}$ that is obtained by making different choices is ud-homotopic to the original $\mC^F$. (The prefix `ud' means `unitality with divisor axiom'.)
If we denote by $\phi'$ the ring isomorphism obtained from $\mC^{F'}$, the formula (\ref{formula_phi_eq}) implies
\[
\phi'(Y^\alpha)=\phi(Y^\alpha) \cdot \exp \Big\langle 
\alpha,
\sum
T^{E(\beta)}  \ (\mC_{0,\beta}^{F'}-\mC_{0,\beta}^F) \ Y^{\partial\beta}
\Big\rangle
\]
and so it suffices to study the \textit{error term}
$S:=\sum
T^{E(\beta)}  \ (\mC_{0,\beta}^{F'}-\mC_{0,\beta}^F) \ Y^{\partial\beta}$
in the exponent.
Indeed, one can exactly use the homotopy between $\mC^F$ and $\mC^{F'}$ to measure the error term, and it turns out that the error term $S$ is governed by the weak Maurer-Cartan equations. So, the two homomorphisms and the two transition maps must agree with each other.
By a similar argument, the cocycle conditions among these transition maps can also be proved.
Ultimately, we can construct a mirror triple $\mathbb X^\vee=(X^\vee,W^\vee,\pi^\vee)$.

\subsubsection{More about choice independence}
\label{sss_application_choice_independence}

A key step in the mirror reconstruction of $\mathbb X^\vee=(X^\vee, W^\vee, \pi^\vee)$ is that we prove the transition map is independent of various choices as explained above.
Accordingly, one can take \textit{specific choices} to extract information.

Let $\mathscr L=(L_t)$ denote the Lagrangian isotopy among the fibers over a path $\sigma$ in $B_0$ with $L_0=\tilde L$ and $L_1=L$. Clearly, the homotopy groups $\pi_2(X,L_t)$ for various $t\in\oi$ can be identified with each other. We often ignore the subtle difference among them.
Then, as before, we have
$\mathcal M_{k,\beta}(J,L_t)
	\cong
	\mathcal M_{k,\beta}(J_t,\tilde L)$.
Moreover, in the definition of the gluing map, the choice-independence (\S \ref{sss_cocycle}) allows us to additionally require the family $\mJ=(J_t)$ is particularly given by $J_t=F_{t*}J$ for a smooth family $F_t\in\diff_0(X)$ between $F$ and $\id_X$ with $F_t(L_t)=L_0$. Thus, the corresponding parameterized moduli space (\ref{moduli_parameterized_eq}) can be identified with $\bigsqcup_t \mathcal M_{k,\beta}(J,L_t)$. From the above discussion, it follows that

\begin{prop}
	\label{trivial_transition_map_prop}
	If every $L_t$ does not bound any Maslov index zero $J$-holomorphic disk, then the gluing map takes the form
	$
	Y^\alpha \mapsto T^{\langle \alpha, \tilde q-q\rangle} Y^{F_*\alpha}$.
\end{prop}

	In comparison, the various polytopal domains in $(\Lambda^*)^n$ are patched together by the maps $Y_i\mapsto T^{c_i}Y_i$. Indeed, there is an admissible covering
	$	(\Lambda^*)^n = \bigcup^\infty_{r= 1}\trop^{-1}(\Delta_r) \equiv \bigcup^\infty_{r= 1} \Sp \Lambda\langle \Delta_r\rangle$
	of polytopal domains for any increasing sequence of rational polyhedrons with $\bigcup_{r\ge 1} \Delta_r=\mathbb R^n$.
	By (\ref{polyhedral_affinoid_algebra_1_eq}), the map $Y_i\mapsto T^{c_i}Y_i$ gives rise to an algebra isomorphism $
	\Lambda\langle \Delta\rangle \to \Lambda\langle \Delta-c\rangle$, hence it induces an isomorphism between
	$\Sp	\Lambda\langle \Delta\rangle \equiv	\trop^{-1}(\Delta)$ and $\Sp 	\Lambda\langle \Delta-c\rangle  \equiv \trop^{-1}(\Delta-c)$ in $(\Lambda^*)^n$.

\begin{cor}
	\label{trivial_shifting_cor}
	Suppose $B_1\subset B_0$ admits an integral affine chart $x: B_1\xhookrightarrow{} \mathbb R^n$. If for every $\tilde q\in B_1$ the Lagrangian fiber $L_{\tilde q}$ bounds no non-constant Maslov-zero $J$-holomorphic stable disk, then there is an isomorphism:
	$
	\pi^\vee|_{B_1}\cong \trop|_{x(B_1)}$.
\end{cor}

\begin{proof}
	Recall that the local charts are the polytopal domains $\Sp\Lambda\langle \Delta,q\rangle\cong \trop^{-1}(\Delta)$.
	By Proposition \ref{trivial_transition_map_prop}, the both sides share the same local charts and transition maps.
\end{proof}

\subsection{Open Gromov-Witten invariants}
\label{ss_OGW}

\subsubsection{Definition}
\label{sss_OGW_defn}

For a Maslov-two class $\beta\in \pi_2(X,L)$, the number $\m_{0,\beta}/\one $ in the superpotential function $W^\vee$ is also known as the (genus-zero one-pointed) \textit{open Gromov-Witten invariant}, denoted by
\begin{equation}
	\label{OGW_eq}
	\mathsf n_\beta := \mathsf n_{\beta,J}
\end{equation}
See e.g. \cite[Definition 2.14]{CLL12} or \cite[Definition 4.4]{lau2014open}.
In \cite{FOOOBookOne}, the value $\mathsf n_{\beta}$ can be regarded as the intersection number of the `$n$-chain'
$\ev_0: \mathcal M_{1,\beta}(J,L)\to L$ with the point class $[pt]$ in $L$.
In general, it depends on both $J$ and other auxiliary data (e.g. the harmonic contraction used for the homological perturbation). But, one can show that if $\beta$ is non-separable (with respect to $J$) in the sense that there does not exist two nontrivial stable disks $u_1,u_2$ with $[u_1]+[u_2]=\beta$, then the number $\mathsf n_\beta$ only depends on $J$. (The proof is tedious but straightforward.) 
Recall that the Maslov index of any holomorphic disk bounded by $L$ is assumed to be non-negative. So, when there does not exist any Maslov-zero $J$-holomorphic stable disk bounded by $L$, any Maslov-two class $\beta$ is non-separable.

\subsubsection{Open GW invariants along a Lagrangian isotopy}
More generally, we study how the open GW invariants evolve along a smooth Lagrangian isotopy $\mathscr L=(L_t)_{t\in\oi}$.
We assume every $L_t$ is oriented, compact and relatively-spin, or we simply assume $L_t$ is a torus.
Note that the relative homotopy groups $\pi_2(X,L_t)$ for various $t\in\oi$ can be identified with each other. Let $\beta=\beta_t\in\pi_2(X,L_t)$ with $\mu(\beta)=2$.
The following result is well known; see e.g. the proof of \cite[Proposition 4.30]{CLL12}.

\begin{thm}
	\label{OGW_inv_isotopy_thm}
	If every $L_t$ does not bound any Maslov index zero $J$-holomorphic stable disk, then the open GW invariants $\mathsf n_{\beta_t, J}$ keep constant along
	the Lagrangian isotopy $\mathscr L=(L_t)$.
	Namely, we have $\mathsf n_{\beta_{t_1}, J}= \mathsf n_{\beta_{t_2}, J}$ for any $0\le t_1,t_2\le 1$.
\end{thm}

\subsection{Toric fibration}
\label{ss_compact_toric}

We want to reinterpret \cite{FOOOToricOne} in the framework of \cite{Yuan_I_FamilyFloer}. It is actualy straightforward.

Let $M\cong \mathbb Z^n$ and $N\cong \mathbb Z^n$ be dual lattices of rank $n$. Define $M_{\mathbb R}=M\otimes \mathbb R$, $N_{\mathbb R}=N\otimes \mathbb R$, and $T_N=N_{\mathbb R}/N$. So, $T_N\cong N_{\mathbb R}/N\cong \mathbb R^n / \mathbb Z^n$.
Suppose $P\subset M_{\mathbb R}$ is a Delzant polyhedron. Let $X$ be the compact toric manifold associated to the normal fan of $P$. Let $\{v_i\}_{i=1}^d$ denote the inner normal vectors of $P$. Consider the linear functionals
$
\ell_i(q)=\langle v_i,  q \rangle -c_i$
$\ell_\infty (q)= \langle  \textstyle  \sum_{i=1}^d v_i, q \rangle $
on $M_{\mathbb R}$.
Then,
$
P=\{ q \in M_{\mathbb R} \mid \ell_i(q)\ge 0 \}
$ and $\partial P=\bigcup_i P_i$ where $P_i$ is the facet of $P$ defined by the zero locus $\ell_i=0$. Following {\cite[Theorem 4.5]{Guillemin1994kaehler}}, we choose the symplectic form to be $ \omega=\frac{\sqrt{-1}}{2\pi} \partial\bar\partial \big(\pi^* (\sum_{i=1}^m (\log \ell_i)+\ell_\infty ) \big)$; see also \cite[\S 2.2]{FOOOToricOne}.

Let $\pi:X\to (\mathbb R^n)^*\cong M_{\mathbb R}$
be the moment map associated to the toric $T^n$-action. Then, $\pi(X)=P$.
The irreducible toric divisor $D_i$ associated to the inner normal vector $v_i$ is
$D_i= \pi^{-1}(P_i)$.
For any point $q$ in the interior $P^\circ$, the fiber
$L_q:=\pi^{-1}(q)$
is a \textit{Lagrangian $T^n$-orbit}.

Let $e_1,\dots, e_n$ be a $\mathbb Z$-bases of $N$ and let $e_1',\dots,e_n'$ be the dual basis of $M$.
Let $S_i^1\cong S^1$ be the subgroup generated by the $i$-th generator element in $T_N\cong (S^1)^n$.
The orbits $S_i^1(q)$ of these subgroups $S_i^1$ in the $T^n$-orbit $L_q$ give rise to a basis of $\pi_1(L_q)\equiv H_1(L_q ;\mathbb Z)$.
Accordingly, we have natural isomorphisms
$
N\cong H_1(L_q ;\mathbb Z)$ and $  M\cong H^1(L_q;\mathbb Z)$. The pairing $ \pi_1(L_q)\otimes H^1(L_q)  \to \mathbb R$ can be also identified with $N \otimes M_{\mathbb R}  \to\mathbb R$.
On the other hand,
$
H_2(X,\pi^{-1}(P^\circ)) \cong H_2(X,L_q)\cong \pi_2(X,L_q) $ has a basis consisting of
$\beta_i\equiv \beta_{i,q} \in \pi_2(X,L_q)$
which is topologically defined by taking a small disk transversal to $D_i$.
Note that
$
\beta_i\cdot D_j=\delta_{ij}$
and
$\partial \beta_i \cong v_i$; c.f. \cite[\S 6]{FOOOToricOne}.
Every $\beta_i$ can be represented by a holomorphic disk
$u_i:(\mathbb D,\partial\mathbb D)\to (X,L_q)$
with $\mu(\beta_i)=2$ and
$E(\beta_{i,q})=  \ell_i(q)$.
The holomorphic disks $u_i$ are Fredholm regular by \cite[Theorem 6.1]{Cho_Oh}.
Denote by $H_2^\eff(X;\mathbb Z)$ the \textit{effective cone}, the cone in $H_2(X;\mathbb Z)$ generated by holomorphic spheres in $X$.

The restriction of the moment map $\pi$ over the interior $\mathrm{Int} P$ of $P$ will be still denoted by $\pi$.
Write $P_q:=\mathrm{Int} P -q \subset \mathbb R^n$ for any $q\in \mathrm{Int} P$.
We can identify $\Lambda[[N]]$ with $\Lambda[[\pi_1(L_q)]]$ for various $q$.
It is well known that $\mathsf n_{\beta_i}=1$ for all $1\le i\le d$ \cite{Cho_Oh}.

\begin{thm}
	\label{Family_Floer_toric_thm}
	Suppose that there is no nontrivial holomorphic sphere with negative Maslov index. 
	Fix $q\in \mathrm{Int} P$.
	The mirror triple $\mathbb  X^\vee_q:=(X_q^\vee, W_q^\vee, \pi_q^\vee)$ associated to $(X,\pi)$ is given by $X_q^\vee \equiv \trop^{-1}(P_q)$, $\pi_q^\vee=\trop: X^\vee_q \to P_q$ and
	\begin{equation}
	\label{W_u_eq}
	\textstyle
	W^\vee_q=\sum_{i=1}^d 
	\Big(
	1+\sum_{\alpha\in H^\eff_2(X;\mathbb Z)\setminus\{0\}} T^{\omega\cap \alpha} \mathsf n_{\beta_i+\alpha}
	\Big) \cdot T^{ \ell_i(q)} Y^{v_i} 
	\end{equation}
	such that for $q_1\neq  q_2$ in the base, we have a natural isomorphism $\mathbb X^\vee_{q_1} \cong \mathbb X^\vee_{q_2}$ given by
	$	Y^v\mapsto T^{  \langle v, q_2-q_1 \rangle}Y^v$. In special, when $X$ is Fano, the superpotential function is given by
	$W_q^\vee= \sum_{i=1}^d T^{\ell_i(q)} Y^{v_i}$.
\end{thm}

\begin{proof}
	\setlength{\parskip}{0em}
	By \cite[Theorem 11.1(5)]{FOOOToricOne}, for every $q\in P^\circ$, the Lagrangian torus fiber $L_q$ bounds no stable disk whose Maslov index is zero.
	Hence, by Corollary \ref{trivial_shifting_cor}, we have $X_q^\vee\cong \trop^{-1}(P_q)$ and $\pi_q^\vee=\trop$.
	Moreover, by \cite[Theorem 11.1]{FOOOToricOne} again, $\mathsf n_\beta\neq 0$ only if $\beta$ is in the form $\beta=\sum_{i=1}^d k_i\beta_i+\alpha$ for some $\alpha\in H_2^\eff (X;  \mathbb Z)$; since $\mu(\beta)=2$, we must have $\beta=\beta_i+\alpha$ with $c_1(\alpha)=0$.
	Recall that $\mathsf n_{\beta_i}=1$. Now, the formula of $W^\vee_q$ in (\ref{W_u_eq}) follows by definition.
	As for the Fano case, 
	just notice that we must have $\beta=\beta_i$ in addition.
	Finally, let $q_1$ and $q_2$ be two distinct points in $\mathrm{Int}P$. Note that $\ell_i(q_2)-\ell_i(q_1)=\langle  v_i, q_2-q_1 \rangle$, and it is direct to check $W_{q_1}^\vee$ is transformed to $W_{q_2}^\vee$.
\end{proof}

For later use, we remark that the above argument actually tells that $\mathsf n_{\beta_j}=1$, and if $\mathsf n_{\beta}\neq 0$, $\beta=\beta_j+\alpha$ for some $1\le j\le d$ and $\alpha \in H_2^\eff(X)$.

\section{Gross's special Lagrangian fibration: Review}
\label{S_Gross_fibration}

\subsection{Notations and descriptions}
\label{ss_notation_construction_Gross}

\subsubsection{Toric data}
\label{sss_toric_data}
Let $\langle ,\rangle$ denote the pairing $N\otimes M \to \mathbb Z$.
Let $\Sigma\subset N_{\mathbb R}$ be a strongly-convex smooth fan, and its rays are $v_0, v_1,\dots, v_d\in N$.
Suppose there exists $m_0\in M$ so that $\langle v_i,m_0\rangle =1$ for all $i=0,1,\dots, d$ and $\langle v,m_0\rangle \ge 1$ for all $v\in|\Sigma| \cap N -\{0\}$.

The (noncompact) toric variety $X=X_{\Sigma}$ defined by the fan $\Sigma$ is Calabi-Yau.
Every ray $v_i$ gives an irreducible toric divisor $D_i$ in $X$. Every $m\in M$ gives a character $\chi^m: T_N\to \mathbb C^*$ on $X$ with
$(\chi^m)=\sum_{i=0}^d \langle v_i, m\rangle D_i$.
Fix a $\mathbb Z$-basis 
$\{e'_0=m_0, e'_1,\dots,e'_{n-1}\}$
of the lattice $M$ and the dual basis $\{e_0=v_0, e_1,\dots,e_{n-1}\}$ of $N$.
The sublattice
$
\tilde N=\{ n\in N\mid \langle n, m_0\rangle =0\} \equiv \mathbb Z\{e_1,\dots,e_{n-1}\}
$ has its dual lattice $\tilde M=\Hom_{\mathbb Z}(\tilde N,\mathbb Z)$ identified with $
M/\mathbb Z m_0
\cong
\{ m\in M \mid \langle v_0,m\rangle =0 \}
\cong
\mathbb Z\{e'_1,\dots,e_{n-1}'\}$.
Let $\pr: M\to \tilde M$ and
$\pr: M_{\mathbb R}\to \tilde M_{\mathbb R}\equiv \tilde M\otimes \mathbb R$ be the projections.
Conventionally, we think
$
M_{\mathbb R}\cong  \tilde M_{\mathbb R}\oplus \mathbb R m_0 \cong \mathbb R \{ e_1' ,\dots, e_{n-1}'\} \oplus \mathbb R e_0' \cong \mathbb R^{n-1}\oplus \mathbb R$
and $
N_{\mathbb R}\cong \tilde N_{\mathbb R}\oplus \mathbb R v_0\cong\mathbb R\{ e_1,\dots, e_{n-1}\}\oplus \mathbb R e_0\cong \mathbb R^{n-1}\oplus \mathbb R$,
where the $e'_0$ or $e_0$ corresponds to the last coordinate rather than the first one.

Consider a polyhedral complex $\mathscr P$ of the convex hull of $\Sigma(1)=\{v_0,\dots,v_d\}$ in the hyperplane $\mathbb R^{n-1}\times \{1\}\equiv \{v\in N_{\mathbb R}\mid \langle v, m_0\rangle =1\}$ such that the set $\mathscr P^{(0)}$ of vertices is precisely $\Sigma(1)$. Hence, $|\Sigma|=\mathbb R_{\ge 0}\cdot  \mathscr P \subset N_{\mathbb R}$.
As in \cite[\S 3.1]{AAK_blowup_toric}, we choose a set of constant real numbers $\{c_0,c_1,\dots,c_d\}$ which is \textit{adapted to $\mathscr P$} in the sense that there is a convex piecewise linear function $\rho$ on the convex hull of $\Sigma(1)$ so that $\rho(v_i)=c_i$ and its maximal domains of linearity are exactly the cells of $\mathscr P$.
Now, we consider the tropical polynomial $\varphi: \tilde M_{\mathbb R} \cong \mathbb R^{n-1} \to \mathbb R$ defined by
$\varphi( \xi)=\max\{ -\langle v_i, \xi \rangle  -c_k \mid k=0,1,\dots, d\}$.
It is a convex function, and its epigraph
$P=
\{ (\eta,\xi)\in M_{\mathbb R} \mid \eta\ge \varphi(\xi)\}
=
\{m\in M_{\mathbb R} \mid \langle v_i,m\rangle \ge c_i \ \forall \ 0\le i\le d \}$
is a convex polytope that defines the fan structure $\Sigma$.

Define 
$P_I=\{m\in P\mid \langle v_i,m \rangle  =c_i, \ \forall \ i\in I\}$
for $I\subset\{0,1,\dots, d\}$, and
then we have $v_i= (-\nabla \varphi, 1)|_{\pr(P_i)}$. By convex analysis, given $p\in \pr(P_i)$ and $p'\in \pr(P_\ell)$, we have 
\begin{equation}
\label{convex_imply_eq}
\langle v_\ell -v_i, p' -p \rangle <0
\end{equation}
Let $\Pi \subset \tilde M_{\mathbb R}$ denote the tropical hypersurface defined by $\varphi$, i.e. the set of points where the maximum is achieved at least twice. (e.g. when $X=K_{\mathbb P^1}$, the set $\Pi$ consists of two points.) In general, we have 
$\Pi=\bigcup_{|I|\ge 2} \pr( P_I)$.
If for any $0\le i\le d$, we denote by $H_i$ the interior of $\pr(P_i)$, then we have
$
\tilde M_{\mathbb R}\setminus \Pi=\bigcup_{i=0}^d H_i =: H$.
We will call $H$ the \textit{wall}. Observe that $\Pi$ gives the dual cell complex of $\mathscr P$.


%

\subsubsection{Moment map}
\label{sss_moment_map}

Let $T_N=N\otimes U(1)\cong T^n$ and $T_N^{\mathbb C}\cong (\mathbb C^*)^n$ denote the real and complex torus respectively.
Take a toric K\"ahler form $\omega$ associated to the polytope $P$. Denote the moment map by
$\mu_X: X\to P\subset M_{\mathbb R}$.
Besides, note that $T_{\tilde N}\cong T^{n-1}$, and we denote the moment map for the $T_{\tilde N}$-action on $X$ by 
$
\tilde\mu_X: X\to \tilde M_{\mathbb R}
$.
Clearly, $\tilde \mu_X= \pr\circ \mu_X$.

For any $0\le i\le d$, the irreducible toric divisor $D_i$ itself is also a toric manifold.
The $\tilde \mu_X$ induces a moment map
$\mu_{D_i} :=\tilde \mu_X|_{D_i} : D_i \to \pr(P_i)\equiv \bar H_i$
for a Hamiltonian $T^{n-1}$-action.
Then, the irreducible toric divisors in the toric manifold $D_i$ are of the form $D_i\cap D_k$ for some $0\le k\le d$.
For clarity, we set $\mathcal I(i)=\{ k \mid D_i\cap D_k\neq \varnothing\}$.
The fiber of $\mu_{D_i}$ over a point in $H_i$ is a Lagrangian $T^{n-1}$-orbit.
Suppose $L'\cong T^{n-1}$ is such a Lagrangian fiber.
Then, we denote by
$
\beta_{ik} \in H_2(D_i, L')\cong \pi_2(D_i, L')$ the classes such that $\beta_{ik} \cdot D_\ell =\delta_{k\ell}$ for any $\ell\in \mathcal I(i)$. Also, every $\beta_{ik}$ can be represented by a holomorphic disk
$
\label{u_ik_D_i_eq}
u_{ik}:  (\mathbb D,\partial\mathbb D) \to ( D_i, L')$ so that
$
u_{ik}\cdot D_k=1$ and $u_{ik}\cdot D_{k'}=0$ for $k'\neq k, \ 0\le k'\le d$.

\subsubsection{Gross's fibration}
\label{sss_Gross_fibration_defn}

Define $w(x)=\chi^{m_0}(x)$
to be the meromorphic function associated to $m_0\in M$.
Fix a constant $\epsilon>0$, and we define $\mathfrak \rho(x)=|w(x)-\epsilon|^2-\epsilon^2$.
Let $B=\tilde M_{\mathbb R}\times \mathbb [-\epsilon^2, \infty)$.
In this paper, the \textit{Gross's fibration} refers to the map
$
\hat\pi: X\to B$
defined by $\hat\pi(x)=(\tilde\mu_X(x), \rho(x))$.
The \textit{discriminant locus} of $\hat\pi$ is given by
$
\Gamma := \partial B \cup ( \Pi \times \{0\})$
(see e.g. \cite[Proposition 4.9]{CLL12}). We set $B_0=B-\Gamma$ and $X_0=\pi^{-1}(B_0)$. The restriction of $\hat\pi$ over $B_0$ yields a smooth Lagrangian torus fibration, denoted by
$
\pi:X_0\to B_0$, which we still call the \textit{Gross's fibration}.
By {\cite{Gross_ex} or \cite[Proposition 4.3 \& 4.7]{CLL12}}, it is a special Lagrangian torus fibration with respect to the holomorphic volume form $(w-\epsilon)^{-1} d\zeta_0\wedge d\zeta_1\wedge \cdots d\zeta_{n-1}$ where $\zeta_i$ is the meromorphic function corresponding to $e_k'\in M$.

Consider the following subsets of $B_0$:
$B_+ = \tilde M_{\mathbb R}\times (0,+\infty)$,
$B_- = \tilde M_{\mathbb R} \times (-\epsilon,0)$,
$H =  \big(\tilde M_{\mathbb R} \setminus \Pi\big) \times \{0\}$, and
$
\partial B = \tilde M_{\mathbb R}\times \{-\epsilon^2\}$.
We will call $H$ the \textit{wall},  since only the Lagrangian fibers over a point in $H$ can bound a nontrivial Maslov zero holomorphic disk (see Lemma \ref{Maslov_zero_locate_lem} below).
For simplicity, we identify $H_i$ (resp. $H$) in $\tilde M_{\mathbb R}$ with $H_i\times \{0\}$ (resp. $H\times \{0\})$ in $M_{\mathbb R}\equiv \tilde M_{\mathbb R}\times \mathbb R$. Then, we have $H=\bigsqcup_i H_i$ and $B_0=B_+\cup H\cup B_-$.
From $\tilde\mu_X(D_i)=\pr (P_i)=\bar H_i$, it follows that $\pi(D_i)=(\tilde \mu_X, \rho)(D_i)=\bar H_i \times \{0\}$.
Finally, we define
$\mathscr E:=\pi^{-1}(\partial B) = \{x\in X\mid w(x)=\epsilon\}$.

In the toric case (\S \ref{ss_compact_toric}), the Lagrangian fibers are the orbits of Hamiltonian $T^n$-action. The Gross's fibers have a weaker symmetry that comes from the action of sub-torus $T_{\tilde N}\cong T^{n-1}$.
For
$q=(q_1,q_2)\in B_0\subset B=\tilde M_{\mathbb R} \times [-\epsilon^2,\infty)$,
the Gross's fiber $L_q=\{ x\in X \mid \tilde\mu_X(x)=q_1, |w(x)-\epsilon|^2=\epsilon^2+q_2\}$
can be viewed as the union of an $S^1$-family of $T_{\tilde N}$-orbits:
$L_q = \bigsqcup_{\theta\in [0,2\pi)} S_q(\theta)$
where $
S_q(\theta)
:=\{ x\in X \mid \tilde\mu_X(x)=q_1, \ w(x)=\epsilon+ e^{i\theta}\sqrt{\epsilon^2+q_2}
\}$
is an orbit of the $T_{\tilde N}$-action.
The following statements are clear from the definition.

\begin{prop}
	\label{L_r_intersect_divisor_prop}
	Fix $c_0\in \mathbb C$. The intersection $L_q\cap w^{-1}(c_0)$ is either the empty set or some $S_q(\theta)$ such that $c_0=\epsilon+e^{i\theta}\sqrt{\epsilon^2+q_2}$. In other words, the holomorphic function $w:X\to \mathbb C$ restricts to a $T^{n-1}$ fibration on $L_q$ over the circle $\{z\in\mathbb  C \mid |z-\epsilon|=\sqrt{\epsilon^2+q_2} \}$.
Therefore, when $c_0\neq 0$, $w^{-1}(c_0)$ is topologically $(\mathbb C^*)^{n-1}$, so $\pi_2(w^{-1}(c_0), L_q\cap w^{-1}(c_0)) \cong \pi_2( (\mathbb C^*)^{n-1}, T^{n-1})=0$. In particular, there is no non-trivial holomorphic disk in $w^{-1}(c_0)$ bounded by $L_q \cap w^{-1}(c_0)$.
\end{prop}

\begin{lem}
	\label{Maslov_zero_locate_lem}
	For $q=(q_1,q_2)\in B_0$, the Gross fiber $L_q=\pi^{-1}(q)$ bounds a non-constant Maslov index zero holomorphic disk if and only if $q_2=0$, i.e. $q\in H$.
	Moreover, the image of a Maslov index zero holomorphic disk bounded by $L_q$ for $q\in H$ is always contained in the divisor $w^{-1}(0)=\bigcup_i D_i$.
\end{lem}
\begin{proof}
	This is basically proved in \cite[Lemma 4.27]{CLL12}; see also \cite[Lemma 5.4]{AuTDual}.
\end{proof}

\begin{cor}
	\label{Maslov_zero_bdry_cor}
	Let $u:(\mathbb D,\partial \mathbb D) \to (X,L_q)$ be a Maslov index zero holomorphic disk. Then, $q\in H$, and the boundary $\partial u: \partial \mathbb D\to L_q$ is contained in $L_q\cap w^{-1}(0)=S_q(\pi)\cong T^{n-1}$. In particular, for any Maslov zero class $\gamma \in\pi_2(X,L_q)$, the evaluation map $\ev: \mathcal M_{1,\gamma}(L_q) \to L_q$ is supported in $S_q(\pi)\cong T^{n-1}$.
\end{cor}

\subsubsection{Lagrangian isotopy}
\label{sss_Lag_isotopy}
Let $q=(q_1,q_2)\in B_0$. Consider the smooth Lagrangian isotopy
$\mathscr L(t):=\mathscr L_q(t):=
\big\{ x\in X \mid \tilde\mu_X(x)=q_1; \quad |w(x)-t|^2=\epsilon^2+q_2 
\big\}$, $ t\in [0,\epsilon]$.
Note that $\mathscr L_q(\epsilon)=L_q=\pi^{-1}(q)$. When $t=0$, we get a Lagrangian $T_N$-orbit
$\mathscr L_q:= \mathscr L_q(0)$.
By \cite[Lemma 4.29]{CLL12},
when $q \in B_+$, all the Lagrangians $\mathscr L_q(t)$ in the isotopy do not admit non-trivial Maslov-zero holomorphic disks. Moreover, when $q \in B_+$, every $\mathscr L_q(t)$ is contained in $T_N^{\mathbb C}\cong (\mathbb C^*)^n$.

On the other hand, given a path $\sigma: \oi \to B_0$, we denote the Lagrangian isotopy among the fibers over $\sigma$ by $\mathscr L_\sigma$.
It gives rise to the isomorphisms
$
\mathscr P_\sigma: \pi_2(X,L_{\sigma(0)})\to \pi_2(X,L_{\sigma(1)})$
and
$
\mathscr P_\sigma : \pi_1(L_{\sigma(0)}) \to \pi_1(L_{\sigma(1)})
$ such that $\mathscr P_\sigma\circ \partial = \partial\circ \mathscr P_\sigma$.

\subsection{Monodromy}
Recall that $T_N^{\mathbb C}\cong (\mathbb C^*)^n$ is an open dense complex torus in $X$.
Recall that $\beta_i$ denotes the class of a small closed disk transversal to the toric divisor $D_i=\mu_X^{-1}(T_i)$.
Then, $\pi_2(X, T_N^{\mathbb C}) = \mathbb Z \langle \beta_0,\beta_1,\dots,\beta_d \rangle$.
The $\pi_1(T_N^{\mathbb C})$ is naturally isomorphic to $N$ in that every $v\in N$ determines a one-parameter subgroup of $T_N^{\mathbb C}$.
For clarity, we always stick to the following specific coordinate: for the previously chosen bases $e_i\in N$ and $e_i'\in M$, we have an isomorphism
\begin{equation}
\label{pi_1_T_N_C}
\pi_1(T_N^{\mathbb C}) \cong \mathbb Z^n
\end{equation}
by declaring that the class $[\gamma]$ for a loop $\gamma\subset T_N^{\mathbb C}$ corresponds to the $n$-tuple 
$\big(\deg (\chi^{e_1'}\circ \gamma), \dots, \deg (\chi^{e_{n-1}'}\circ \gamma), \deg (w\circ \gamma)\big)$ in $\mathbb Z^n$.
A useful observation is that 
\begin{equation}
\label{degree(w-e)_eq}
\deg (w\circ \gamma)=
\begin{cases}
\deg (w\circ \gamma-\epsilon) \qquad & \text{when $\gamma \subset \pi^{-1}(B_+)$} \\
0 \qquad & \text{when $\gamma\subset \pi^{-1}(B_-\cup \partial B)$}
\end{cases}
\end{equation}

Next, we consider the boundary homomorphism
$
\partial: \pi_2(X, T_N^{\mathbb C})\equiv \mathbb Z^{d+1} \to \pi_1(T_N^{\mathbb C})\equiv \mathbb Z^n$.
We write $v_i^k:=\langle v_i, e'_k \rangle \in \mathbb Z$ for $0\le k\le n-1$; particularly $v_i^0=\langle v_i,m_0\rangle =1$ for $0\le i\le d$.
Then, $v_i=v_i^0 e_0+\cdots v_i^{n-1}e_{n-1}\in N$.
The principal divisor is $(\chi^{e'_k})= \sum_{i=0}^d v_i^k D_i$.
By definition, the class $\partial\beta_i$ only goes around the toric divisor $D_i$ once; thus, we have $\deg ( \chi^{e_k'} \circ \partial\beta_i)= v_i^k$ and $\deg ( w \circ \partial\beta_i)= v_i^0=1$. To sum up, respecting the identification $\pi_1(T_N^{\mathbb C})= \mathbb Z^n$ in (\ref{pi_1_T_N_C}), we have
$
\partial\beta_i= (v_i^1,v_i^2,\dots, v_i^{n-1}, 1) \in \mathbb Z^n$.
Furthermore, we can take an identification $\mathbb Z^n\cong N$ via $(m_1,\dots,m_{n-1},m_0)\mapsto m_1e_1+\cdots +m_{n-1}e_{n-1}+m_0e_0$ for which we exactly obtain $\partial\beta_i=v_i$. So, we often abuse the notations and write $v_i=(v_i^1,\dots,v_i^{n-1},1)$.
In special,
$v_0=e_0=(0,\dots,0,1)
$.

For a loop in $B_0$ that only passes through the two wall components $H_i$ and $H_j$ once, one can show that the monodromy is given by the matrix
\begin{equation}
\begin{bmatrix}
1 & 0 & \cdots &\cdots & 0& v_j^1- v_i^1 \\
& 1 & \cdots &\cdots & 0 & v_j^2-v_i^2 \\
&    & \ddots &   &\vdots & \vdots \\
&	  &             &    1     & 0  & v_j^{n-2}-v_i^{n-2} \\
&	  &             &         & 1  & v_j^{n-1}-v_i^{n-1} \\
& & & & & 1
\end{bmatrix}_{n\times n}
\end{equation}

When $q \in B_+$ or $q \in B_-$, we have an inclusion map
$j_q: L_q \xhookrightarrow{} T_N^{\mathbb C} \subset X$ which induces a group homomorphism
$
j_{q*}: \pi_1(L_q)\to  \pi_1(T_N^{\mathbb C})\equiv \mathbb Z^n
$.
By (\ref{pi_1_T_N_C}), the $k$-th coordinate of $j_{q*}[\gamma]$ for a loop $\gamma\subset L_q$ is $\deg (\chi^{e'_k}\circ \gamma)$ for $1\le k\le n-1$, and its $n$-th coordinate is $\deg (w\circ \gamma)$.
Also, the inclusion map
$s_q^\theta: S_q(\theta)\xhookrightarrow{} L_q$
gives a group embedding
$s_{q*}^\theta:
\pi_1(S_q(\theta)) \xhookrightarrow{} \pi_1(L_q)$.

\begin{lem}
	\label{s_q_theta_describe_lem}
	Fix $\theta\in[0,2\pi)$. For any $q\in B_+$ or $q\in B_-$, 
	we have $\im ( j_{q*} \circ s_{q*}^\theta) = \mathbb Z^{n-1}\times \{0\}$. Especially, 	the class $\partial (\beta_i-\beta_j)$ is contained in the image of the embedding $\pi_1(S_q(\theta))\xhookrightarrow{}\pi_1(L_q)$.
\end{lem}

\begin{proof}
	Take a set $\{\gamma_1,\dots,\gamma_{n-1}\}$ of loops in $S_q (\theta)$ which give rise to a basis of $\pi_1(S_q(\theta))\cong \im s_{q*}^{\theta}\subset \pi_1(L_q)$.
	Note that $\tilde  N=\mathbb Z \{ e_1,\dots, e_{n-1}\}$, and we may assume that for each $1\le i\le n-1$, the loop $\gamma_i$ is the $S^1$-orbit of the action of $S^1\cdot e_i  \subset  T_{\tilde N}\equiv T^{n-1}$.
	For an arbitrary loop $\gamma$ in $S_q(\theta)$, the $w\circ \gamma$ is constant, and so $\deg w\circ \gamma=0$.
	By (\ref{pi_1_T_N_C}), we conclude $\im( j_{q*}\circ s_{q*}^\theta) = \mathbb Z^{n-1}\times \{0\}$.
\end{proof}

\begin{prop}
	\label{j_q_describe_prop}
	When $q\in B_+$, we have an isomorphism $j_{q*}:\pi_1(L_q)\to\mathbb Z^n$.
	When $q\in B_-$, the $j_{q*}$ is an embedding such that $\im j_{q*}=\mathbb Z^{n-1}\times \{0\}$; in particular, the rank of the kernel of $j_{q*}$ is one.	 
	Moreover, given a path $\sigma$ in $B_+$ (resp. in $B_-$) between $q'=\sigma(0)$ and $q''=\sigma(1)$, the isomorphism $\mathscr P_\sigma: \pi_1(L_{q'})\to \pi_1(L_{q''})$ induced by the Lagrangian isotopy $\mathscr L_\sigma$ satisfies that $j_{q''*}\circ \mathscr P_\sigma= j_{q'*}$.
\end{prop}

\begin{proof}
	\setlength{\parskip}{0em}
	By Lemma \ref{s_q_theta_describe_lem}, it suffices to understand the degree of a loop $\gamma'\subset L_q$ that meets every sub-torus $S_q(\theta)$ exactly once in order to completely describe the $j_{q*}$. Indeed, it is easy to see that $w\circ \gamma'$ forms the circle $|z-\epsilon|=\sqrt{\epsilon^2+q_2}$ in $\mathbb C^*$. So, we know $\deg (w\circ \gamma')$ is equal to $\pm 1$ if $q_2>0$ and equal to $0$ if $q_2<0$. Moreover, note that the maps $\chi^{e'_k}$ and $w$ have no zeros or poles in $\pi^{-1}(B_+)$ or $\pi^{-1}(B_-)$. So, the $\deg (\chi^{e'_k}\circ \gamma')$ and $\deg (w\circ \gamma')$ keep unchanged along the Lagrangian isotopy.
\end{proof}

\begin{defn-prop}
	\label{hat_beta_wp_defn}
	When $q\in B_-$, there exists a generator, denoted by $\hat\beta=\hat\beta_q$, in the rank-one group $\pi_2(T_N^{\mathbb C}, L_q)\subset \pi_2(X,L_q)$ such that $\hat\beta\cdot \mathscr E=1$.
\end{defn-prop}

\begin{proof}
	Denote $\mathbb D'=\{z\in\mathbb C\mid |z-\epsilon|\le \sqrt{\epsilon^2+q_2}\}$, and the holomorphic function $w$ restricts to a $T^{n-1}$-fibration on $L_q$ over the circle $\partial \mathbb D'$ (Proposition \ref{L_r_intersect_divisor_prop}). Then, we take a section $\gamma:\partial \mathbb D'\to L_q$ of $w|_{L_q}$ over $\partial \mathbb D'$ which by definition intersects every $S_q(\theta)$ exactly once. By Lemma \ref{s_q_theta_describe_lem} and Proposition \ref{j_q_describe_prop}, we have $j_{q*}(\gamma)=0$, that is, $\gamma$ is contractible in $T_N^{\mathbb C}$. Hence, we can fill it into a disk $s:\mathbb D'\to T_N^{\mathbb C}$ such that $\partial s=\gamma$.
	By deforming the disk $s$ within $T_N^{\mathbb C}=(\mathbb C^*)^n$, we may assume $w\circ s=\id$. Since $\mathscr E=w^{-1}(\epsilon)$ and $\epsilon \in\mathbb D'$, the topological intersection number $[s]\cdot \mathscr E$ is either $+1$ or $-1$.
	Finally, we define $\hat\beta$ to be either $[s]$ or $-[s]$ such that $\hat\beta\cdot \mathscr E=1$.
\end{proof}

For clarity, we use $q_+$ and $q_-$ to represent some base points in $B_+$ and $B_-$ respectively. Also, we often abbreviate $L_+=L_{q_+}$ and $L_-=L_{q_-}$.
Recall that a path $\sigma$ with $\sigma(0)=q_-$ and $\sigma(1)=q_+$ induces an isomorphism
$\mathscr P_\sigma : \pi_2(X,L_-)\to \pi_2(X,L_+)$.
Fix $0\le i\le d$.

\begin{lem}
	\label{pi2_monodromy_lem}
	Let $\sigma_i$ be a path in $B_0$ that passes through the wall component $H_i$. Then $
	\mathscr P_{\sigma_i} (\hat\beta) =\beta_i$.
\end{lem}

\begin{proof}
	\setlength{\parskip}{0em}
	Assume $\mathscr P_{\sigma_i}(\hat\beta)= k_0\beta_0 +\cdots +k_d\beta_d$, where $k_i\in\mathbb Z$. Since the Lagrangian isotopy $\mathscr L_{\sigma_i}$ is totally contained in $\pi^{-1}(B_+\cup H_i\cup B_-)$ and $D_\ell$ is contained in $\pi^{-1}(\overline{H_\ell\times \{0\}})$, the intersection number of a disk class with $D_\ell$ is preserved along $\mathscr L_{\sigma_i}$ whenever $\ell\neq i$. Thus, we have
	$
	\hat \beta \cdot D_\ell = \mathscr P_{\sigma_i}(\hat\beta) \cdot D_\ell$ for any $\ell\neq i$.
	Clearly, $\hat\beta\cdot D_\ell=0$. As $\beta_k \cdot D_\ell=\delta_{k\ell }$, we conclude that $k_\ell=0$ for any $\ell\neq i$. Hence, $\mathscr P_{\sigma_i}(\hat\beta)=k_i\beta_i$ for some $k_i\in\mathbb Z$.
	Moreover, the Lagrangian isotopy preserves the Maslov index, which implies that $2k_i=\mu(k_i\beta_i)=\mu(\hat\beta)=2$. Thus, $k_i=1$. The proof is now complete.
\end{proof}


\subsection{Mirror construction}

\subsubsection{Set-up}
Denote by $(X^\vee,W^\vee,\pi^\vee)$ the family Floer mirror of $(X,\pi)$ (\S \ref{sss_family_Floer_general}).
By Lemma \ref{Maslov_zero_locate_lem}, the Gross's fibration $\pi$ restricted over $B_+$ or $B_-$ does not encounter any Maslov-zero disk. Then, by Corollary \ref{trivial_shifting_cor},
$
X_\pm^\vee:= (\pi^\vee)^{-1}(  B_\pm) \cong \trop^{-1}(  B_\pm)$.
Let $W_\pm^\vee$ denote the restrictions of $W^\vee$ over the two chambers $X^\vee_\pm$.
The (one-pointed genus-zero) open GW invariant is denoted by $\mathsf  n_\beta= \mathsf n_{\beta,J}$ (\S \ref{ss_OGW}).
We review several results as below which are known before; see e.g. \cite{AuTDual} or \cite{CLL12}.

\begin{lem}
	\label{OGW_Gross_B_+_lem}
	Suppose $q \in B_+$. If $\mathsf n_{\beta}\neq 0$,
	then there exists $\alpha\in H^\eff_2(X)$ and $0\le j\le d-1$ such that $\beta=\beta_j+\alpha$.
	Moreover, we have $\mathsf n_{\beta_j}=1$.
\end{lem}

\begin{proof}
	The aforementioned Lagrangian isotopy $\mathscr L_q(\cdot)$ (\S \ref{sss_Lag_isotopy}) induces $\mathscr P: \pi_2(X,L_q)\cong \pi_2(X, \mathscr L_q)$.
	By Theorem \ref{OGW_inv_isotopy_thm}, $\mathsf n_\beta= \mathsf n_{\mathscr P(\beta)}$.
	So, it reduces to the toric case (end of \S \ref{ss_compact_toric}).
\end{proof}


\begin{lem}
	\label{holo_disk_B_-_lem}
	Fix $q\in B_-$. A non-trivial holomorphic disk $u:(\mathbb D,\partial\mathbb D)\to (X,L_q)$
	is contained in
	$S_-:=\pi^{-1}(B_-\cup\partial B)=\{x\in X \mid |w(x)-\epsilon|<\epsilon \}$.
	Moreover, there exists some $k>0$ so that $[u]=k\hat\beta$.
\end{lem}

\begin{proof}
	Write $q=(q_1,q_2)$ and $\mathbb D'=\{ z\in \mathbb C\mid |z-\epsilon|\le \sqrt{\epsilon^2+q_2}\}$ as before. We observe that $w\circ u: (\mathbb D,\partial\mathbb D)\to (\mathbb D',\partial\mathbb D')$ is a holomorphic map such that $w\circ u(\partial\mathbb D)\subset \partial\mathbb D'$.
	By maximum principle, $|w\circ u-\epsilon|\le \sqrt{\epsilon^2+q_2}<\epsilon$ on $\mathbb D$. Thus, the image of $u$ is always contained in $S_-$, which implies that $u$ cannot meet any toric divisor $D_i$, as $w^{-1}(0)=\bigcup_i D_i$. It follows that the class $[u]$ lives in the subgroup $\pi_2(T_N^{\mathbb C},L_q)\cong \mathbb Z \cdot \hat\beta$ (Definition-Proposition \ref{hat_beta_wp_defn}), say $[u]=k\hat\beta$.
	It suffices to show $k>0$.
	Since the $u$ is non-trivial, the $w\circ u$ cannot be a constant function due to Proposition \ref{L_r_intersect_divisor_prop}. So, we have $w\circ u(\mathbb D)=\mathbb D'$.
	Accordingly, the $u$ has at least one intersection point with the divisor $\mathscr E=w^{-1}(\epsilon)$. Moreover, since $u$ is holomorphic, any intersection is positive. So, $0<[u]\cdot \mathscr E= k \hat\beta\cdot \mathscr E=k$.
\end{proof}

\begin{lem}
	\label{holo_sphere_lem}
	Any non-trivial holomorphic sphere $h:\mathbb P^1\to X$ is contained in $w^{-1}(0)=\bigcup_i D_i$. In particular, the sphere $h$ does not intersect with $S_-$.
\end{lem}

\begin{proof}
	The holomorphic map $w\circ h: \mathbb P^1 \to \mathbb C$ must be constant, say $w\circ h=c_0$.
	In other words, the sphere $h$ is contained in the level set $w^{-1}(c_0)$.
	If $c_0 \neq 0$, then the level set $w^{-1}(c_0)$ is topologically $(\mathbb C^*)^{n-1}$. But, we know $\pi_2( (\mathbb C^*)^{n-1})=0$, which implies $h$ must be trivial.
	So, we have $c_0=0$, and thus the image of $h$ is contained in $w^{-1}(0)=\bigcup_i D_i$.
\end{proof}

\begin{lem}
	\label{OGW_Gross_B_-_weak_lem}
	If $q\in B_-$ and $\beta\in \pi_2(X,L_q)$ with $\mu(\beta)=2$, then
	$\mathsf n_{\beta}\neq 0$
	only if $\beta=\hat\beta$.
\end{lem}

\begin{proof}
	Suppose $\uu$ is a holomorphic stable disk in the class $\beta$ that contributes to $\mathsf n_\beta$.
	Since $X$ is Calabi-Yau, discarding sphere bubbles in $\uu$ does not affect the Maslov index.
	So, by Lemma \ref{holo_disk_B_-_lem}, there is at most one non-trivial disk component, denoted by $u:(\mathbb D,\partial\mathbb D)\to (X,L_q)$, such that $[u]=\hat\beta$, and all the sphere bubbles have total Maslov index zero.
	By Lemma \ref{holo_disk_B_-_lem} again, the image of $u$ is contained in $S_-$.
	However, by Lemma \ref{holo_sphere_lem}, any non-trivial holomorphic sphere does not meet $S_-$. Thereby, the $\uu$ cannot possess any sphere bubbles, and $\beta=[\uu]=[u]=\hat\beta$. The proof is now complete.
\end{proof}

\subsubsection{Superpotentials}
By {\cite[Proposition 4.32]{CLL12}}, we can show $\mathsf n_{\hat\beta}=1$ by explicitly showing the existence and uniqueness of the holomorphic disk in the class $\hat\beta$.
We would like to give an alternative proof of this fact by using the family Floer theory, i.e. Theorem \ref{wall_cross_identity_intro_thm}.
Recall that $W_\pm^\vee$ denotes the restriction of $W^\vee$ on the chamber $X^\vee_\pm$. By Lemma \ref{OGW_Gross_B_+_lem}, using the fact that $\mathsf n_{\beta_j}=1$ \cite{Cho_Oh} yields
\begin{equation}
\label{W_+_Gross_general_eq}
\textstyle
W_+^\vee = \sum_{j=0}^d T^{E(\beta_j)} Y^{\partial \beta_j}
\Big(
1+\sum_{ 0\neq \alpha\in H^\eff_2(X)} T^{E(\alpha)} \mathsf n_{\beta_j+\alpha}
\Big)
\end{equation}
Besides, by Lemma \ref{OGW_Gross_B_-_weak_lem}, we have
$
W_-^\vee= T^{E(\hat\beta)} Y^{\partial\hat\beta} \mathsf n_{\hat\beta}$.
Beware that the various classes $\beta_j$'s in the expression of $W_+^\vee$ actually depend on a base point $q\in B_+$. But, the difference between the local expression of $W_+^\vee$ over another base point $q'$ and the original one is exactly made up by the map $Y^\alpha\mapsto T^{\langle\alpha, q'-q\rangle}Y^\alpha$. The similar thing holds for $W_-^\vee$.
Fix $0\le i\le d$, and we take the path $\sigma=\sigma_i$ in $B_0$ that passes through the component $H_i$ so that $\sigma(0)=q_-\in B_-$ and $\sigma(1)=q_+\in B_+$.
We know
$\mathscr P_\sigma(\hat\beta)=\beta_i$
by Lemma \ref{pi2_monodromy_lem}, but for simplicity we would rather write
$\hat\beta =\beta_i$ in view of the Fukaya's trick.
Adopting this convention, we also define
$\gamma_\ell:=\beta_\ell-\hat \beta$ for $0\le \ell \le d, \ \ell \neq i$.

\subsubsection{Maslov index zero classes}
It is clear that the topological class $\gamma_\ell$ has Maslov index zero.
The subgroup $\pi_2'(X,L_0)$ in $\pi_2(X,L_0)\cong \mathbb Z^{d+1}$ consisting of Maslov index zero classes has codimension one and is spanned by all $\gamma_\ell$ ($\ell\neq i$). In particular, we have $\pi_2'(X,L_0) \cong \mathbb Z^d$ by associating to the tuple $\pmb k:=(k_0,k_1,\dots, k_{i-1},k_{i+1},\dots,k_d) \in\mathbb Z^d$ the class 
$
\gamma_{\pmb k}:=\sum_{\ell\neq i, \ 0\le \ell \le d} k_\ell\gamma_\ell
$.
We write $\pmb k\ge 0$ if $\pmb k\in\mathbb Z_{\ge 0}^d$, i.e. all $k_\ell\ge 0$; we write $\pmb k>0$ if $\pmb k\ge 0$ and $\pmb k\neq (0,\dots,0)$.
If the class $\gamma_{\pmb k}$ is represented by a non-trivial Maslov index zero holomorphic stable disk, then
$E(\gamma_{\pmb k}) > 0$ and $\pmb k>0$. 
We speculate that the converse is also true. However, we do not bother to explicitly find those Maslov index zero holomorphic disks;
the following weaker statement is enough for our purpose. (Remark that Lemma \ref{energy_inequality_lem} below is not essential and can be checked by hand in the concrete examples.)

\begin{lem}
	\label{energy_inequality_lem}
	In the above situation, for any $\ell\neq i$, we have $E(\beta_\ell)>E(\hat\beta)$, i.e. $E(\gamma_\ell)>0$.
\end{lem}

\begin{proof}
	To begin with, we claim that the lemma holds if $D_\ell\cap D_i\neq \varnothing$ for all $\ell\neq i$. In fact, we notice that the intersection $L':=L_0\cap D_i=L_0\cap w^{-1}(0)$ is exactly an orbit of the action of $T_{\tilde N}\cong T^{n-1}$ by Proposition \ref{L_r_intersect_divisor_prop}.
	So, as in \S \ref{sss_moment_map}, we can find a holomorphic disk $u$ contained in $D_i$ so that $u\cdot D_\ell=1$ and $u\cdot D_{\ell'}=0$ for a fixed $\ell$ and any $\ell\neq \ell' \neq i$. By considering intersection numbers, one can easily check that $[u]=\gamma_\ell=\beta_\ell-\hat\beta$. Since $E(u)>0$, we conclude that $E(\gamma_\ell)=E(\beta_\ell)-E(\hat\beta)>0$.
	
	In general, we first note that the purpose is a strict inequality. Thus, by Fukaya's trick, we may work with the Lagrangian fiber $L=L_+$ over a base point $q=q_+$ which is slightly above the wall. Specifically, we set $q=(q_1, \delta)\in\tilde M_{\mathbb R}\times \mathbb R$ for $q_1\in H_i$ and a sufficiently small $\delta>0$.
	Choose $q_1'\in H_\ell$, and draw a straight line segment $\mathfrak L$ in $\tilde M_{\mathbb R}\cong \mathbb R^{n-1}$ between $q_1$ and $q_1'$. Perturbing $q_1'$ slightly if necessary, we may assume $\mathfrak L$ is transversal to $\Pi$; then, suppose $\mathfrak L$ meets the wall components $\mathfrak L$ by $H_{\ell_0=i}, H_{\ell_1}, H_{\ell_2}, \dots, H_{\ell_{m-1}}$ and $H_{\ell_m=\ell}$ successively.
	Consider
	\[
	\beta_{\ell,q}-\hat\beta_q= \textstyle \sum_{k=0}^{m-1} (\beta_{\ell_{k+1},q}-\beta_{\ell_k,q})
	\]
	where for clarity we specify the base point $q$.
	It suffices to prove $E(\beta_{\ell_{k+1},q}) > E(\beta_{\ell_k,q})$ for any $0\le k\le m-1$.
	To show this, we pick up a point $q^k_1$ in $H_{\ell_k}\cap \mathfrak L$ and set $q^k=(q_1^k,\delta)\in B_+$. 
	Notice that $H_{\ell_k}$ and $H_{\ell_{k+1}}$ are adjacent along the line segment $\mathfrak L$, thus, $D_{\ell_k}\cap D_{\ell_{k+1}}\neq\varnothing$. Then, using the claim at the start of proof, we conclude that $E(\beta_{\ell_{k+1},q^k})>E(\beta_{\ell_k,q^k})$.
	Consider the line segment $\mathfrak L_+=(\mathfrak L, \delta)$ in $B_+$ between $q=(q_1,\delta)$ and $q':=(q_1',\delta)$, and then the point $q^k$ lies in $\mathfrak L_+$. Note that the line segment $\mathfrak L_+$ has the parametrization $q(t)=( (1-t) q+tq', \delta)$.
	Now, we claim that the function
	\[
	f: \oi \to\mathbb R,\quad  t\mapsto E(\beta_{\ell_{k+1}, q(t)})-E(\beta_{\ell_k,q(t)} )
	\]
	is decreasing in $t\in\oi$.
	In reality, we first have $E(\beta_{q(t_2)})-E(\beta_{q(t_1)})=\langle \partial\beta, q(t_2)-q(t_1)\rangle
	=
	(t_2-t_1) \cdot	\langle \partial\beta, (q'-q) \rangle$.
	It follows that $f'(t)=\langle \partial\beta_{\ell_{k+1}}-\partial\beta_{\ell_k} ,q'-q\rangle= \langle v_{\ell_{k+1}}-v_{\ell_k}, q'-q\rangle$. By (\ref{convex_imply_eq}), we deduce that $f'(t)<0$ and prove the claim.
	Now, the above claim implies $E(\beta_{\ell_{k+1},q})-E(\beta_{\ell_k,q})=f(0)> E(\beta_{\ell_{k+1},q^k}) -E(\beta_{\ell_k,q^k})>0$, which completes the proof.
\end{proof}

\subsubsection{Formal power series identities}
\label{sss_formal_series_ientity}

By Fukaya's trick we may assume the gluing map $\phi: \Lambda[[\pi_1(L_+)]]\to\Lambda[[\pi_1(L_-)]]$across the wall $H$ takes the form:
$Y^\alpha\mapsto Y^{\alpha} \exp \langle \alpha, \pmb F(Y)\rangle$,
where $\alpha \in \pi_1(L_+)\cong \pi_1(L_-)$ and $\pmb F=\sum_{\gamma\neq 0} \mC_\gamma T^{E(\gamma)} Y^{\partial\gamma}$. By definition, the coefficients of $\pmb F$ are given by the counts of Maslov-zero holomorphic disks and are all contained in $\Lambda_+$.
Note that the inverse $\phi^{-1}$ also takes this form.

\begin{thm}
	\label{wall_cross_identity_intro_thm}
	Fix $0\le i\le d$. There exists formal power series $F_k \ (0\le k\le d), G \in \Lambda[[\pi_1(L_q)]] \cong \Lambda[[Y_1^\pm,\dots, Y_n^\pm]]$ which only depend on the counts of Maslov index zero holomorphic stable disks and whose coefficients are all contained in $\Lambda_+$ such that
	\begin{equation}
	\label{wall_cross_1_intro_eq}
	\textstyle
	\mathsf n_{\hat\beta} = N_i \cdot \exp (F_i) + \sum_{\ell\neq i} T^{E(\gamma_\ell)} Y^{\partial \gamma_\ell} \cdot N_\ell \cdot \exp(F_\ell)
	\end{equation}
	and
	\begin{equation}
	\label{wall_cross_2_intro_eq}
	\textstyle
	\mathsf n_{\hat\beta}\cdot \exp(G) = N_i+ \sum_{\ell\neq i} T^{E(\gamma_\ell)} Y^{\partial \gamma_\ell} N_\ell
	\end{equation}
	where $E(\gamma_\ell)=\omega\cap \gamma_\ell>0$ and we define 
	$N_k :=\mathsf n_{\beta_k}+ \sum_{\alpha\neq 0} T^{E(\alpha)} Y^{\partial\alpha} \mathsf n_{\beta_k+\alpha}\in \Lambda_0$.
\end{thm}

\begin{proof}
	Because the transition maps of the mirror space must match $W_+^\vee$ and $W_-^\vee$, we have $W_-^\vee=\phi(W_+^\vee)$ and $W_+^\vee=\phi^{-1} (W_-^\vee)$. By routine computations with (\ref{W_+_Gross_general_eq}), they produce the two identities in (\ref{wall_cross_1_intro_eq}) and (\ref{wall_cross_2_intro_eq}) respectively.
	Moreover, we have $E(\gamma_\ell)>0$ by Lemma \ref{energy_inequality_lem}; also, $E(\alpha)>0$ for any nonzero $\alpha\in H^\eff_2(X)$.
	The proof of Theorem \ref{wall_cross_identity_intro_thm} is now complete.
\end{proof}

In either equation, comparing the energy-zero parts yields that
$\mathsf n_{\hat\beta}=\mathsf n_{\beta_i}$.
Moreover, as $\beta_i$ is a toric disk, we know $\mathsf n_{\beta_i}=1$ by \cite{Cho_Oh}. Hence, we give an alternative proof of the following fact:

\begin{cor}[{\cite[Proposition 4.32]{CLL12}}]
	\label{Wall_crosssing_intr_cor}
	$\mathsf n_{\hat\beta}=1$.
\end{cor}

What matters in Theorem \ref{wall_cross_identity_intro_thm} is not the existence of those formal power series but the geometric meaning that they are contributed by the counts of Maslov-zero disks.
This line of ideas inspires the main result of this paper.

\bibliographystyle{alpha}
\bibliography{mybib}

\end{document}